\documentclass[12pt,reqno]{amsart}
\usepackage[top=1in, bottom=1in, left=1in, right=1in]{geometry}
\usepackage{times, amsthm, amssymb, amsmath, amsfonts, bm, graphicx,setspace}
\usepackage{mathrsfs,wasysym}
\usepackage[pagebackref=true,pdftex]{hyperref}
\usepackage[all]{xy}
\usepackage[usenames,dvipsnames]{color}
\usepackage{latexsym}
\usepackage{amscd}
\usepackage{float}
\usepackage{tikz}
\usepackage{caption}
\usepackage{enumitem}

\numberwithin{equation}{section}

\newtheorem{theorem}{Theorem}[section]
\newtheorem{proposition}[theorem]{Proposition}
\newtheorem{lemma}[theorem]{Lemma}
\newtheorem{corollary}[theorem]{Corollary}

\theoremstyle{definition}
\newtheorem{definition}[theorem]{Definition}
\newtheorem{convention}[theorem]{Convention}
\newtheorem{example}[theorem]{Example}
\newtheorem{remark}[theorem]{Remark}

\newcommand\<{\langle}
\newcommand\CC{{\mathbb C}}
\newcommand\NN{{\mathbb N}}
\newcommand\QQ{{\mathbb Q}}
\newcommand\RR{{\mathbb R}}
\newcommand\ZZ{{\mathbb Z}}
\newcommand\PP{{\mathbb P}}
\newcommand\cE{{\mathcal E}}
\newcommand\sC{{\mathscr C}}
\newcommand\sL{{\mathscr L}}
\newcommand\cS{{\mathcal S}}
\newcommand\cU{{\mathcal U}}
\newcommand\scrU{{\mathscr U}}
\newcommand\cV{{\mathcal V}}
\newcommand\cW{{\mathcal W}}

\newcommand*{\defeq}{\mathrel{\vcenter{\baselineskip0.5ex \lineskiplimit0pt
                     \hbox{\scriptsize.}\hbox{\scriptsize.}}}%
                     =}
\newcommand*{\eqdef}{
=%
\mathrel{\vcenter{\baselineskip0.5ex \lineskiplimit0pt
                     \hbox{\scriptsize.}\hbox{\scriptsize.}}}%
                     }

\newcommand\del{\partial}

\newcommand{\Arg}{\operatorname{Arg}}

\newcommand\charVar{{\rm Char}}

\newcommand\GL{\mathrm{GL}}
\newcommand\gr{{\rm gr}}

\newcommand\New{{\rm New}}

\newcommand\rank{{\rm rank}}
\newcommand\Res{{\rm Res}}
\newcommand\ResArr{{\rm ResArr}}
\newcommand\Sol{{\rm Sol}}
\newcommand\Pol{{\rm Poles}}

\newcommand\Sing{{\rm Sing}}
\newcommand\vol{{\rm vol}}
\newcommand\Var{{\rm Var}}

\newcommand\ini{{\rm in}}
\newcommand\fin{{\rm fin}}
\newcommand\minus{\smallsetminus}

\renewcommand\>{\rangle}

\renewcommand\Re{{\rm Re}}

\def\endrk{\hfill$\hexagon$}

\def\fface{\left[\begin{smallmatrix}1\\0\end{smallmatrix}\right]}
\def\lface{\left[\begin{smallmatrix}1\\k\end{smallmatrix}\right]}
\def\4face{\left[\begin{smallmatrix}1\\4\end{smallmatrix}\right]}

\newsavebox{\smlmat}
\savebox{\smlmat}{$\left[\begin{smallmatrix}1&1&1\\0&2&3\end{smallmatrix}\right]$}

\newsavebox{\smlMat}
\savebox{\smlMat}{$\left[\begin{smallmatrix}1&1&1&1\\0&1&3&4\end{smallmatrix}\right]$}

\begin{document}
\vspace*{-9mm}
\title{Hypergeometric functions for projective toric curves}

\dedicatory{In memory of Andrei Zelevinsky}

\author{Christine Berkesch Zamaere}
\address{School of Mathematics, University of Minnesota, Minneapolis, MN 55455.}
\email{cberkesc@math.umn.edu}

\author{Jens~Forsg{\aa}rd}
\address{Department of Mathematics \\ Stockholm University \\
SE-106 91 Stockholm, Sweden.}
\email{jensf@math.su.se}

\author{Laura Felicia Matusevich}
\address{Department of Mathematics \\
Texas A\&M University \\ College Station, TX 77843.}
\email{laura@math.tamu.edu}

\thanks{
CBZ was partially supported by NSF Grants DMS 1440537 and OISE 0964985. 
JF was partially supported by the G. S. Magnusson Fund of the Royal Swedish Academy of Sciences. 
LFM was partially supported by NSF Grants DMS 0703866 and DMS 1001763, and a Sloan Research Fellowship.}
\begin{abstract}
We produce a decomposition of the parameter space of the
$A$-hypergeometric system associated to a projective monomial curve as
a union of an arrangement of lines and its complement, in such a way that the
analytic behavior of the solutions of the system is 
explicitly controlled within each term of the union.
\end{abstract}
\maketitle
\vspace{-6mm}
\section{Introduction}
\label{sec:intro}

The goal of this article is to understand the behavior, as the
parameters vary, of the solutions of the $A$-hypergeometric system
associated to a projective monomial curve.  
Our strongest result is Theorem~\ref{thm:bigSheaf}, which asserts the
existence of finitely many lines in the parameter space $\CC^2$,
outside of which the solutions of the corresponding hypergeometric
system vary locally analytically. Further, these hypergeometric
functions are also locally analytic separately along those lines, and the
behavior at line intersections is completely understood. In
particular, this provides
an intuitive picture for the formation of rank jumps in such a system.

One of the interesting features of  hypergeometric systems of differential
equations is
that changing the parameters can have significant effects. For
instance, special choices of parameters can result in the presence of
algebraic, rational or polynomial solutions. This has motivated the
study of hypergeometric functions as functions of the parameters:
the classical hypergeometric series and integrals, such as those named
after Euler, Gauss, Appell, and Lauricella, are known to be
meromorphic functions of the parameters.  
In the more general context of $A$-hypergeometric differential
equations, hypergeometric functions representable as Euler or
Euler--Mellin integrals share this property~\cite{GKZ90,NP}.  
In all of these cases, the parameters at which the generic solutions
have poles
yield hypergeometric systems that are difficult to solve explicitly 
and may present additional exceptional behavior. 
One such major challenge is that the dimension of the solution space
(called the \emph{holonomic rank}) of the hypergeometric system can
change as its parameters vary.  

While the holonomic rank of a hypergeometric system has been well
studied~\cite{MMW,berkesch}, the most general results in this
direction are obtained using algebraic tools that are completely
divorced from the solution spaces whose dimensions they compute.  
In order to completely understand hypergeometric functions as the
parameters vary, rank computations need to be realized in terms of
solutions.  

In general, it is fairly straightforward to
produce linearly independent $A$-hypergeometric functions whose
parametric behavior is 
well controlled, and which, for sufficiently generic parameters, span
the solution space of the corresponding system;  
this is done in Section~\ref{sec:PD+sheaf} using the classical
technique of parametric differentiation. However, dealing with rank
jumping parameters presents significant difficulties (see
Remark~\ref{rmk:higherDimensions}).  
For the remainder of this article, we thus focus on the special case of
$A$-hypergeometric systems arising from projective monomial curves, where
dedicated tools are available.  
Specifically, three facts come to our aid for such systems: 
they have only finitely many rank jumping parameters~\cite{CDD}; for
generic parameters, their Euler--Mellin integral solutions span the
solution space of the system~\cite{BFP}; and for all parameters,
the solutions of the system can be expressed as power series without
logarithms~\cite{SST, saito-logFree}.  
We deeply rely on the combinatorics of this specific situation to
refine the results about hypergeometric integrals and series in order
to explain  
how deformations of special solutions along certain subspaces of the
parameter space produce rank jumps.

Throughout this article, we assume that the toric varieties underlying
the $A$-hypergeometric systems under consideration are
projective, because in this case, integral representations are
available for $A$-hypergeometric functions. 
While affine toric varieties also give rise to
$A$-hypergeometric systems, the corresponding $A$-hypergeometric
integrals are not well understood; in order for our methods to apply
to the affine case,~\cite[Conjecture~5.4.4]{SST}, which provides a
basis of $A$-hypergeometric integrals in the generic affine case,
would have to be proved. Even for affine toric curves, which are
arithmetically Cohen--Macaulay, and therefore have no rank-jumping
parameters by~\cite{MMW}, this conjecture is still open. Consequently, our arguments do not carry over to this case.
It is the subject of ongoing work to produce analogs of our main
results for higher dimensional and affine toric varieties;
Remarks~\ref{rmk:EMobstacles}, ~\ref{rmk:higherDimensions}, 
and~\ref{rmk:seriesObstacles} address further difficulties in
pursuing this construction.  

Let $A$ be a $d\times n$ integer matrix such that $\ZZ A$, the group generated by the columns of $A$, is equal to $\ZZ^d$. 
We also require that $(1,\dots,1)\in\ZZ^n$ is in the $\QQ$-row span of
$A$. This yields a hypergeometric system that has regular singlarities~\cite{hotta}. 
Let $\CC^n$ have coordinates $x = (x_1,\dots,x_{n})$, and let $\del =
(\del_1,\dots,\del_{n})$, where $\del_i \defeq \del/\del x_i$.  
Let $D$ denote the \emph{Weyl algebra} over $\CC^n$, which is the free
associative algebra generated by the variables
$x,\del$ modulo the two-sided ideal   
\[
\left\< x_i x_j = x_j x_i,\ 
\del_i\del_j = \del_j\del_i, \ 
\del_ix_j = x_j\del_i +\delta_{ij}
\mid i,j\in\{1,\dots,n\} \right\>,
\] 
where $\delta_{ij}$ is the Kronecker $\delta$-function. 
Let 
\begin{equation}
\label{eqn:toric}
I_A\defeq \<\del^u-\del^v\mid Au = Av \> \subseteq \CC[\del_1,\dots,\del_n]
\end{equation}
be the \emph{toric ideal} arising from $A$, and consider the \emph{Euler operators} 
\[
E-\beta \defeq 
\left\{\sum_{j=1}^{n} a_{i,j} x_j\del_j - \beta_i\right\}_{i=1}^d\subseteq D.
\] 
The \emph{$A$-hypergeometric system} associated to $A$ and 
parameter $\beta\in\CC^d$ is the left $D$-ideal 
\begin{align*}
H_A(\beta) \defeq  
D\cdot 
\left(
I_A + \< E-\beta\>
\right).
\end{align*}

An $A$-hypergeometric system associated
to a projective toric curve is determined by a matrix 
\begin{equation}
\label{eqn:A}
A = \begin{bmatrix} 
1 & 1 & 1 & \cdots & 1 & 1\\
0 & k_2 & k_3 & \cdots & k_{n-1} & k
\end{bmatrix},
\end{equation}
where $0 = k_1 < k_2 <\cdots < k_{n-1} < k_{n} = k$, and $\gcd(k_2,\dots,k_{n-1},k)=1$.
The curve alluded to above is the one obtained by taking the Zariski
closure in $\PP_{\CC}^{n-1}$ of the image of the map
\[
(\CC^*)^2 \to (\CC^*)^n \quad \text{given by} \quad
(s,t) \mapsto (s , s t^{k_2},\dots,s t^{k_{n-1}}, s t^k) .
\]
The defining equations of this curve are~\eqref{eqn:toric}.
As an illustration of the particular nature of $H_A(\beta)$ when $A$ has the form~\eqref{eqn:A}, recall that for
$\beta = -
\left[\begin{smallmatrix}0\\1\end{smallmatrix}\right]$, the roots of
the polynomial 
\[
f(z) \defeq x_1+x_2z^{k_2}+\cdots+x_{n-1}z^{k_{n-1}}+x_nz^k,
\] 
considered as functions of $x_1,\dots,x_n$, are solutions
of $H_A(\beta)$~\cite{birkeland,mayr,sturmfels-solving}. 

The \emph{singular locus} $\Sigma_A\subseteq\CC^n$ 
of $H_A(\beta)$ is the hypersurface defined by the
principal $A$-determinant, also known as the full $A$-discriminant, see~\cite{GKZa}.
This locus is independent of $\beta$~\cite{gkz88,GKZ,BMW'}. When $A$ has the form~\eqref{eqn:A}, 
$\Sigma_A$ is the union of $\Var(x_1x_n)$ and the hypersurface defined by
the \emph{sparse discriminant} or \emph{$A$-discriminant}, 
which is the irreducible polynomial in $x$ that vanishes whenever $f(z)$ has a repeated root in $\CC^*$.

A \emph{solution} $\Phi$ of $H_A(\beta)$ is a multivalued, locally
analytic function, defined in a neighborhood of a fixed nonsingular point in $\CC^n$, such that $P\bullet \Phi = 0$ for each $P\in H_A(\beta)$. 
We consider functions $\Phi = \Phi(x, \beta)$
that provide solutions of $H_A(\beta)$ for each $\beta$ in some domain.
The solutions of an $A$-hypergeometric system can be represented as series~\cite{GKZ,CDD,SST}, Euler-type integrals~\cite{GKZ90}, or
Barnes-type integrals~\cite{N,B}, with each representation having different
advantages and drawbacks. In this article, we make use of both \emph{extended
  Euler--Mellin integrals}~\cite{NP,BFP} and 
\emph{series solutions}. While the former are entire in the
parameters, 
they do not span the solution space of $H_A(\beta)$ for all $\beta$
(see, for instance, Theorem~\ref{thm:genericResPolar}).  
On the other hand, the solution space of $H_A(\beta)$ can always be spanned by series, but the
domains of convergence in $x$ of these series 
depend not only on the parameters, but also on the monomial ordering
used to perform the series expansion~\cite[Theorem~2.5.16]{SST}.

The combinatorics of the columns of $A$ play a role in our results. 
A \emph{face} $F$ of $A$, is a subset of the column set of $A$ corresponding to a face of the real positive cone $\RR_{\geq 0}A$ over the columns of $A$. 
For $A$ as in \eqref{eqn:A}, 
the faces of $A$ are $A$, $\left\{\fface\right\}$, $\left\{\lface\right\}$, and $\varnothing$.

\begin{definition}
\label{def:resonant}
Let $F$ be a proper face of $A$. 
A parameter $\beta\in\CC^d$ is said to be \emph{resonant with respect to $F$} if $\beta\in\ZZ^d+\CC F$. 
Further, $\beta$ is \emph{resonant} if it is resonant with respect to some proper face of $A$. 
Let $\ResArr(A)$ denote the set of resonant parameters for $A$, so that 
\begin{equation}
\label{eqn:ResArr}
\ResArr(A) = \bigcup_{F\precneqq A} \ZZ^d+\CC F.
\end{equation}
\end{definition}

Resonance is linked to the behavior of the \emph{rank} of $H_A(\beta)$, that is, the
dimension of its solution space.
If $\beta \notin
\ResArr(A)$, then by~\cite{GGZ,GKZ,adolphson},
$\rank(H_A(\beta))=\vol(A)$, the \emph{normalized} volume of $A$. (For $A$ as
in~\eqref{eqn:A}, $\vol(A)=k$.)
Let 
\[
\cE_A \defeq \{\beta\in\CC^2\mid \rank\, H_A(\beta)>\vol(A)\}
\] 
denote the set of \emph{rank-jumping} parameters of $H_A(\beta)$. 
In the case of projective toric curves, the groundbreaking article~\cite{CDD} carried out a series of difficult
computations to show that  
\begin{equation}
\label{eqn:EAforCurves}
\cE_A = [ (\NN A + \ZZ\fface) \cap (\NN A + \ZZ \lface) ] \minus \NN A
\end{equation}
and that $\rank(H_A(\beta)) = k+1$ for $\beta \in \cE_A$. 
The solutions computed in \cite{CDD} are specific to integer parameters. 
An alternative proof of these facts, containing ideas that are crucial to 
Section~\ref{sec:SeriesSolutions}, can be found in~\cite[Section~4.2]{SST}. 

In the case of projective toric curves, while $\cE_A$ consists only of finitely many points, 
the fact that the dimension of the solution space of $H_A(\beta)$ is
constant almost everywhere does not mean that the solutions themselves cannot exhibit
different behaviors on $\CC^2 \minus \cE_A$.
These differences are already evident upon examination of the
solutions computed in~\cite{CDD}. 
Understanding the parametric variation of solutions of $H_A(\beta)$ is the 
motivation for our main results. 

For $A$ as in~\eqref{eqn:A}, let $\bar{\beta} \in \CC^2$, and let $L \subset \CC^2$ be a line
containing $\bar{\beta}$. We will say that a solution $\varphi(x)$ of
$H_A(\bar{\beta})$ which is locally analytic on $\CC^n\minus \Sigma_A$
is \emph{analytically deformable to a solution of $H_A(\beta)$ along
  $L$} if there exists a function $\Phi(x,\beta)$ that is locally analytic on
$(\CC^n \minus \Sigma_A) \times L$ such that $\Phi(x,\beta)$ is a
solution of $H_A(\beta)$ for all $\beta \in L$, and also such that the
specialization of $\Phi(x,\beta)$ to $\beta=\bar{\beta}$ is $\varphi(x)$.

\begin{theorem}
\label{thm:main}
The following assertions hold for $A$ as in~\eqref{eqn:A}.
\begin{enumerate}[leftmargin=*]
\item The solutions of $H_A(\beta)$ are locally analytic functions of
  $x$ and $\beta$ on $(\CC^n \minus \Sigma_A) \times (\CC^2\minus \ResArr(A))$. 
\item If $L$ is a line contained in $\ResArr(A)$, then the solutions of
  $H_A(\beta)$ are locally analytic functions of 
  $x$ and $\beta$ on $(\CC^n \minus \Sigma_A) \times L$.
\item Let $\bar{\beta} \in \ResArr(A)$ be the intersection of two
  lines $L_1$ and $L_2$ contained in $\ResArr(A)$.
\begin{enumerate}
\item Every solution of $H_A(\bar{\beta})$ is analytically deformable
  to a solution of $H_A(\beta)$ both along $L_1$ and along $L_2$ if and
  only if $\bar{\beta} \notin \cE_A$. 
\item If $\bar{\beta} \in \cE_A$, then the solution space of
  $H_A(\bar{\beta})$ can be decomposed as a direct sum of three
  subspaces. The first one is of dimension $1$ and consists of the functions
  that are only analytically deformable to solutions of $H_A(\beta)$ along $L_1$. The second one also
  is of dimension $1$ and consists of the functions that are only analytically
  deformable to solutions of $H_A(\beta)$ along $L_2$. The
  final one is of dimension $\vol(A)-1$ and consists of the functions that are
  analytically deformable to solutions of $H_A(\beta)$ both along $L_1$ and
  along $L_2$.
\end{enumerate}
\end{enumerate}
\end{theorem}

In Theorem~\ref{thm:main}, $\ResArr(A)$ can be replaced by a union of
finitely many resonant lines, as follows.

\begin{theorem}
\label{thm:bigSheaf}
Let $A$ be as in~\eqref{eqn:A}. There exists a set $\sL\subseteq \CC^2$
which is a union of finitely many (polar, see~\eqref{eqn:polar})
resonant lines, and which includes all resonant lines meeting $\cE_A$,
such that every statement in 
Theorem~\ref{thm:main} is still valid when $\ResArr(A)$ is replaced by $\sL$.
\end{theorem}

\subsection*{Outline}
We review necessary facts about extended Euler--Mellin integrals in
Section~\ref{sec:extend}.  
In Section~\ref{sec:PD+sheaf}, 
we use parametric derivatives and known facts about $A$-hypergeometric
integrals to provide a
less refined version of Theorems~\ref{thm:main}
and~\ref{thm:bigSheaf}, which is valid for any $A$-hypergeometric
system (arising from a projective toric variety) that has finitely
many rank jumps; this is 
Theorem~\ref{thm:SolutionsAnalyticStratification}. 
The remainder of this article provides an in-depth study of the
parametric variation of 
$A$-hypergeometric functions associated to projective toric curves.
Sections~\ref{sec:EM:genericResonant}
and~\ref{sec:EM:nongenericResonant} study
the behavior of the extended 
Euler--Mellin integrals. 
A careful analysis of the parametric 
convergence of $A$-hypergeometric series appears in Section~\ref{sec:SeriesSolutions}. 
Theorems~\ref{thm:main} and~\ref{thm:bigSheaf} are then proven in
Section~\ref{sec:sheafProof}, where we also
revisit the link between the solutions at rank-jumping
parameters and a certain local cohomology module.   
\endrk

\subsection*{Acknowledgements}%
We thank Alicia Dickenstein, Pavel Kurasov, Ezra Miller, Timur
Sadykov, and Uli Walther for fruitful conversations related to this
work.  We are grateful to the anonymous referee, whose thoughtful
comments and suggestions have improved this article.
The idea that extended Euler--Mellin integrals could be used to show
that hypergeometric functions vary nicely with the parameters  
was a product of discussions with Mikael Passare at the Institut
Mittag-Leffler in 2011. We wish that this could have been a joint
project with him.  
\endrk


\section{Extended Euler--Mellin integrals}
\label{sec:extend}

In this section, we introduce a tool for computing solutions of $H_A(\beta)$, extended Euler--Mellin integrals. We also recall that these integrals are linearly independent for very general $\beta\in\CC^2$. 

Given a polynomial $f\in\CC[z]$ with $f(0)\neq 0$, its \emph{coamoeba} $\sC(f)$ is the
image of $\Var(f)\subseteq \CC^*$ under the argument map $\arg\colon \CC^*\to\RR/2\pi\ZZ$. 
For a univariate polynomial this is a finite set, and hence closed, in contrast to the general situation.
The \emph{Euler--Mellin integral} associated to a connected component
$\Theta$ of the complement of the (closure of the) coamoeba $\sC(f)$ is
\begin{equation}
\label{eq:EMintegral}
M_f^\Theta(x,\beta) = \int_{\Arg^{-1}(\theta)} 
{f(z)^{\beta_1}}{z^{-\beta_2}}\,
\,\frac{dz}{z},
\end{equation}
where $\theta\in \Theta$ is arbitrary.

In this article, we are interested in the polynomial $f(z) = \sum_{i=1}^{n} x_i z^{k_i}$ for a generic point $x \in \CC^n$. 
The Newton polytope $\New(f)$ has two facets, $\{0\}$ and $\{k\}$, with corresponding (outward) normal vectors $\mu_0 = -1$ and $\mu_k = 1$. It thus admits a representation as an intersection of halfspaces
\[
\New(f) = \{y\in\RR \mid \mu_0 y \leq 0 
,\ \mu_k y -k \leq 0\} 
= \{y\in\RR \mid \gamma_j \cdot (1,y) \leq 0 \text{ for } j \in\{ 0, k\}\},
\]
where $\gamma_0 \defeq (0,-1)$ and $\gamma_k \defeq (-k,1)$. 
It was shown in~\cite[Theorem 2.3]{BFP} that the integral~\eqref{eq:EMintegral} converges in the tubular domain
\begin{equation}
\label{eqn:ConvergenceDomain}
\left\{\beta\in \CC^2 \mid \Re(\gamma_0\cdot \beta) >0 \,\text{ and }\, \Re(\gamma_k\cdot \beta) >0\right\}.
\end{equation}

Viewed as a function of the parameter $\beta$, the function $M_f^\Theta(x,\beta)$ can be  meromorphically extended to $\CC^2$; this is the essence of Hadamard's \emph{partie finie}~\cite{Had}, as understood by Riesz~\cite{Rie}. 
The explicit description of the process under which the meromorphic extension is obtained,
as provided by~\cite{BFP}, makes use of combinatorial information that
is crucial for our description of the behavior of extended Euler--Mellin integrals. 

In~\cite[Theorem 2.5]{BFP} it is stated that, if $\beta$ is contained in \eqref{eqn:ConvergenceDomain}, then
\begin{equation}
\label{eqn:PhiExtension}
M_f^\Theta(x,\beta) = \frac{\Phi_f^\Theta(x, \beta)}{\Gamma(-\beta_1)}\, \Gamma(-\gamma_0\cdot\beta)\,\Gamma(-\gamma_k\cdot\beta)
\end{equation}
where $\Phi$ is entire in $\beta$. Hence, the right hand side of \eqref{eqn:PhiExtension} provides a meromorphic extension of $M_f^\Theta$ in $\beta$. However, to give the most natural formulations of several results in this article, we use the more explicit version of $\Phi$ whose construction is outlined in~\cite[Remarks 2.6--7]{BFP}. 
This makes use of the monoids 
\begin{equation}
\label{eq:numSemigroups}
G_0  = \NN \, \{ k-k_{n-1}, \dots, k-k_2, k \}
\quad \text{and} \quad 
G_k  = \NN \, \{ k_2, \dots, k_{n-1}, k \}
\end{equation}
consisting of nonnegative integer combinations of the listed positive integers.

\begin{theorem}{\cite[Theorem~4.2]{BFP}}
\label{thm:BFP:EMconverge}
For each $x\in (\CC^n\minus\Sigma_A)$ and  each component $\Theta$ of the complement of $\sC(f)$, there
exists a function $\Psi_f^\Theta(x, \beta)$ which is entire in $\beta$ and locally analytic at
$(x,\beta)\in (\CC^n\minus\Sigma_A)\times\CC^2$, such that  
\begin{align}
\label{eq:extendedEM}
M^\Theta_f(x,\beta) = \frac{\Psi_f^\Theta(x, \beta)}{\Gamma(-\beta_1)}\, \Gamma(-\gamma_0\cdot\beta)\,\Gamma(-\gamma_k\cdot\beta)\, 
\prod_{\kappa\in\NN\minus G_0} (-(\gamma_0\cdot\beta)-\kappa)
\prod_{\kappa\in \NN\minus G_k}(-(\gamma_k\cdot\beta) - \kappa),
\end{align}
whenever $\beta$ belongs to \eqref{eqn:ConvergenceDomain}.
\end{theorem}

\begin{definition}
\label{def:extendedEM}
We call the function $\Psi_f^\Theta(x, \beta)$ 
in~\eqref{thm:BFP:EMconverge} the \emph{extended Euler--Mellin integral} corresponding to $f, \beta$, and the component $\Theta$ of the complement of the coamoeba $\sC(f)$. 
\end{definition}

The \emph{polar locus} of $A$, denoted by $\Pol(A)$, is the union in $\CC^2$ of 
the zero loci of the linear polynomials 
$-(\gamma_0\cdot\beta)-\kappa$, for $\kappa\in G_0$ 
and 
$-(\gamma_k\cdot\beta) - \kappa$, for $\kappa \in G_k$. 
In other words, 
\begin{equation}
\label{eqn:polar}
\Pol(A) = 
\bigcup_{\kappa\in G_0} \Var(-(\gamma_0\cdot\beta)-\kappa)
\ \cup \  
\bigcup_{\kappa\in G_k} \Var(-(\gamma_k\cdot\beta) - \kappa).
\end{equation}

Note that $\Pol(A)$ is strictly contained in the set of resonant parameters, $\ResArr(A)$ of~\eqref{eqn:ResArr}, see Figure~\ref{fig:PolesVsResonance}.
It contains the poles of all Euler--Mellin integrals $M_f^\Theta(x,\beta)$ by Theorem~\ref{thm:BFP:EMconverge}. 
Note also that $\Phi_f^\Theta(x, \beta)$ vanishes along the (finitely
many) lines, corresponding to the poles of the Gamma functions in~\eqref{eqn:PhiExtension} which are not contained in 
$\Pol(A)$; however, this vanishing is artificial as it is introduced by the use of unnecessary factors in the Gamma functions 
in the definition of $\Phi_f^\Theta$. 
This explains our preference for the functions $\Psi_f^\Theta$. 

\begin{figure}[t]
\centering
\includegraphics[angle=180,width=50mm]{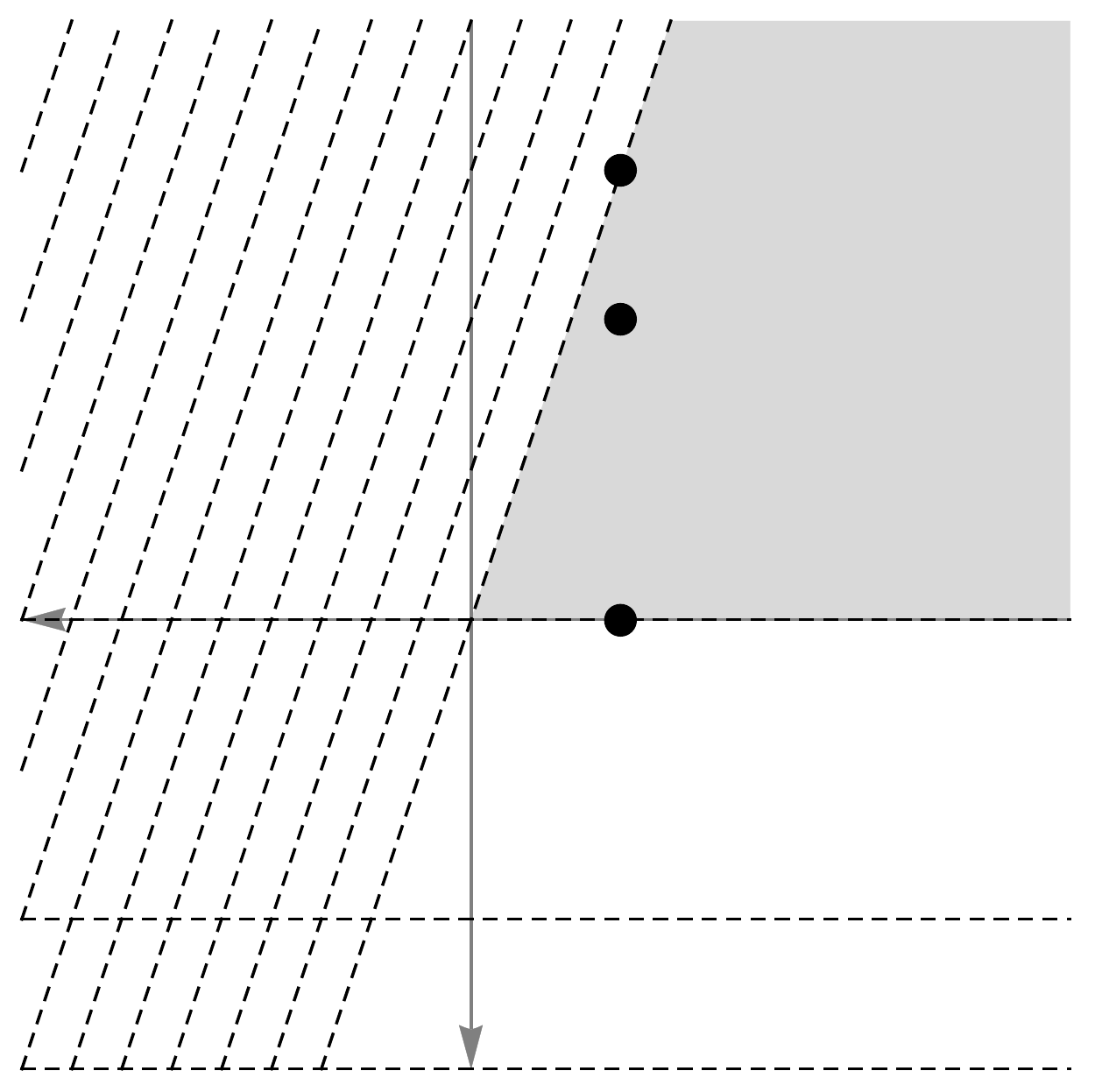}
$\qquad$
\includegraphics[angle=180,width=50mm]{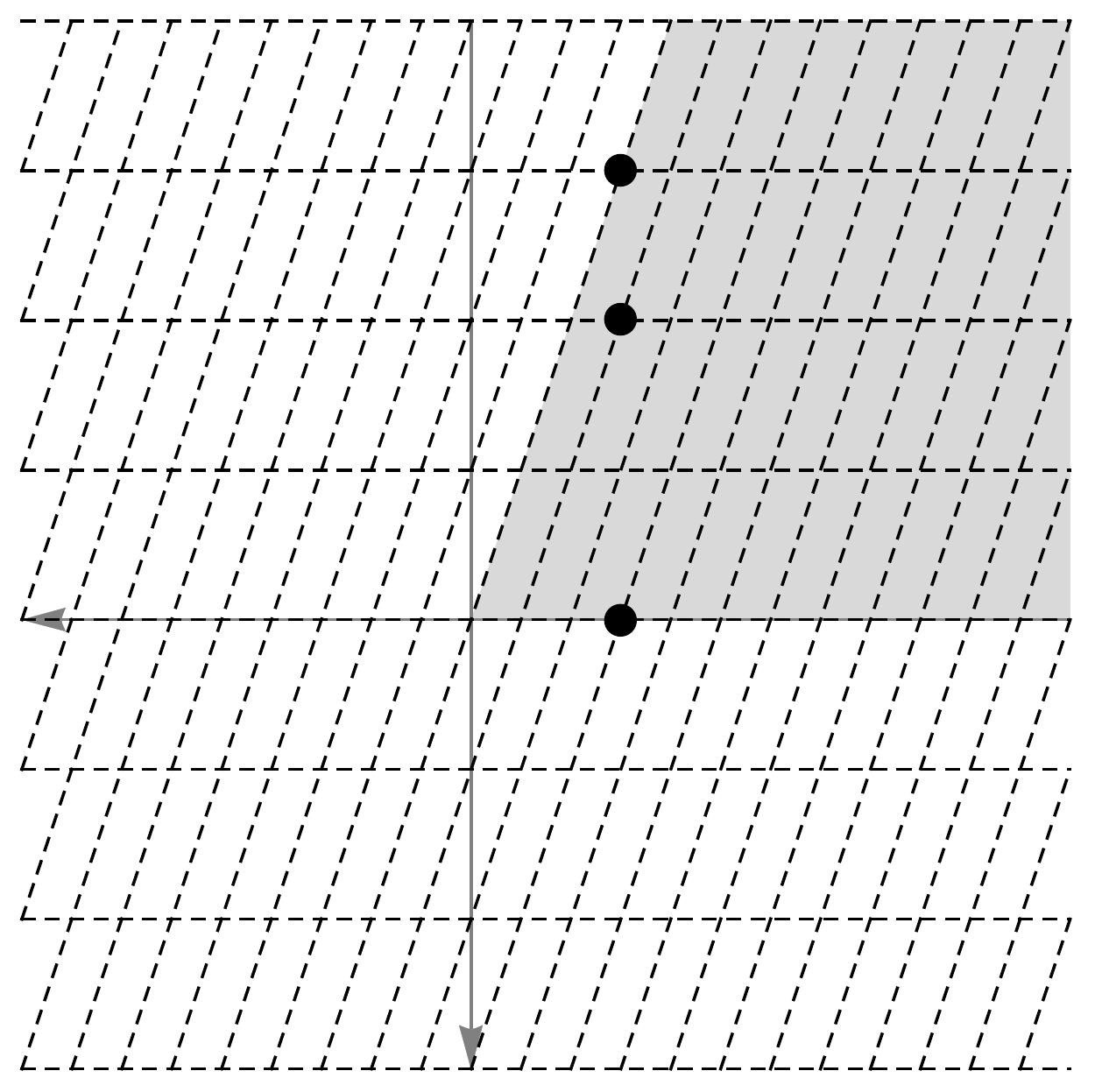}
\caption{The polar lines (left) and resonant lines (right) when
$A = $\usebox{\smlmat}.
The opposites of the points of $A$
are highlighted. The shaded region is the domain
of convergence of $M_f^\Theta(x,\beta)$.} 
\label{fig:PolesVsResonance}
\end{figure}

Let $a_i$ denote the $i$th column of the matrix $A$, i.e., $a_i = \left[\begin{smallmatrix}1\\k_i\end{smallmatrix}\right]$. 
The meromorphic extension of~\eqref{eq:EMintegral} is in this notation given by the
\emph{extension formulas} from~\cite[(2.9)]{BFP}: 
\begin{align}
\label{eqn:ExtensionFormulas-1}
M_f^\Theta(x,\beta) 
= & \frac{-\beta_2}{\gamma_0\cdot\beta} \sum_{i = 1}^{n} k_i\,
x_{i}\,M_f^\Theta(x,\beta- a_i)\quad\text{and} \\
\label{eqn:ExtensionFormulas-2}
M_f^\Theta(x,\beta) 
= & \frac{-\beta_2}{\gamma_k\cdot\beta}\sum_{i=1}^{n}(k-k_i)\,x_{i}\,M_f^\Theta(x,\beta-a_i).
\end{align}
The factor $k_i$ in~\eqref{eqn:ExtensionFormulas-1} implies that the right hand side has a vanishing term for $i=1$, and similarly for $i=n$ in~\eqref{eqn:ExtensionFormulas-2}. 
This vanishing implies that \eqref{eqn:ExtensionFormulas-1} and \eqref{eqn:ExtensionFormulas-2} provide meromorphic extensions 
of $M_f^\Theta(x,\beta)$; the right hand side of~\eqref{eqn:ExtensionFormulas-1} converges in the domain
\[
\bigcap_{i=2}^{n} \left\{\beta\in \CC^2 
\;\big\vert\;
\Re(-\gamma_0\cdot(a_i+\beta)) < 0 \,\text{ and }\,
\Re(-\gamma_k\cdot(a_i+\beta)) < 0
\right\},
\]
which, since $k_i > 0$ for $i>0$, is strictly greater
than~\eqref{eqn:ConvergenceDomain}. 

Applying the extension formula~\eqref{eqn:ExtensionFormulas-1} a second time, the $i$th term of~\eqref{eqn:ExtensionFormulas-1} is divided by the linear form 
\[
\gamma_0\cdot( \beta - a_i) = (\gamma_0\cdot\beta) + \kappa,
\]
where $\kappa = -\gamma_0\cdot a_i$ is a nonzero element of $G_0$. 
Repeated use of the extension formulas \eqref{eqn:ExtensionFormulas-1} and \eqref{eqn:ExtensionFormulas-2} yields an expression for the meromorphic extension of $M_f^\Theta(x, \beta)$ as a linear combination of shifted Euler--Mellin integrals, with
coefficients that are 
products of reciprocals of distinct linear forms corresponding to resonant lines that belong to $\Pol(A)$.

Consequently, an expression for the extended Euler--Mellin integral $\Phi_f^\Theta(x, \beta)$ can, for each fixed $\beta$, be obtained from 
$M_f^\Theta(x,\beta)$ by iterations of the extension
formulas~\eqref{eqn:ExtensionFormulas-1}
and~\eqref{eqn:ExtensionFormulas-2}. 
To do this, 
choices must be made as to the order in which the two types of extensions are performed. 
We focus on the two orderings given by either first extending over lines resonant with respect to the face $\{\fface\}$, and then extending over lines resonant with respect to the face $\{\lface\}$, or vice versa.

If $\beta$ is contained in a polar line $L$ with defining equation $(\gamma_j\cdot\beta) - \kappa  = 0$, then restricting the entire function $\Psi_f^\Theta(x, \beta)$ in $\beta$ to $L$, the only terms in the iterated expansion of $M_f^\Theta(x,\beta)$ that do not vanish are those for which the linear form $(\gamma_j\cdot\beta) - \kappa$ appears in the denominators of their coefficients. 
Thus, in order to explain the behavior of $\Psi_f^\Theta$ at such resonant parameters, it is necessary to carefully track the combinatorics of the coefficients in the expansion process of $M_f^\Theta(x,\beta)$. 
This is of particular importance when the parameter $\beta$ is contained in the intersection of two polar lines.

We conclude this section by noting that 
when $A$ has the form~\eqref{eqn:A} and $\beta$ is nonresonant (see Definition~\ref{def:resonant}), then extended Euler--Mellin integrals form a basis of the solution space of $H_A(\beta)$ at any nonsingular $x\in\CC^n$. 

\begin{theorem}{\cite[Proposition~5.1]{BFP}}
\label{thm:EMindep}
The extended Euler--Mellin integrals $\{\Psi_f^\Theta(\beta,x)\}$ of Definition~\ref{def:extendedEM}, where $\Theta$ ranges over a set of connected components of the complement of $\sC(f)$, are linearly independent when $\beta \in \CC^2 \minus \ResArr(A)$ and 
\[
x \in 
V \defeq \{ x \in \CC^n \mid \max(|x_2|, \dots, |x_{n-1}|) \ll \min(|x_1|, |x_n|)\}\subseteq \CC^n\minus \Sigma_A. 
\]
Hence, for any 
$\beta \in \CC^2 \minus \ResArr(A)$, such a collection of extended Euler--Mellin integrals can be analytically continued to form a basis of the solution space of the $A$-hypergeometric system $H_A(\beta)$ at any $x\in\CC^n\minus \Sigma_A$. 
\end{theorem}

\begin{remark}
\label{rmk:EMobstacles}
If $A$ is more general, so that it 
corresponds to a $(d-1)$-dimensional projective toric variety, where $d>2$, we are not aware of a proof that the set of integrals obtained from the components of the complement of the closure of $\sC(f)$ are linearly independent. 
If this were the case, it still might happen that $A$ does not have $\vol(A)$-many associated extended Euler--Mellin integrals that could be used to form a basis of the solution space of some $H_A(\beta)$. Even if one works up to $\GL_d(\ZZ)$-equivalence, which induces an isomorphism on $A$-hypergeometric systems, this difficulty cannot always be overcome. 
For example, if $A$ is the matrix
\begin{equation}
\label{eqn:A-ex}
A = \begin{bmatrix} 
1 & 1 & 1 & 1 & 1\\
0 & 2 & 2 & 1 & 0\\
0 & 0 & 2 & 3 & 3
\end{bmatrix},
\end{equation}
then the rank of $\Sol_{\bar{x}}(H(\beta))$ is $11$ for all parameters $\beta$. However,
$10$ is the maximal number of connected components of the complement of the coamoeba of a polynomial
\[
\hspace{4cm}
f(z_1,z_2) = x_1 + x_2 z_1^2 + x_3 z_1^2 z_2^2 + x_4 z_1 z_2^3 + x_5 z_2^3.
\hspace{4cm}
\hexagon
\]
\end{remark}

\section{Parametric derivatives and solution sheaves}
\label{sec:PD+sheaf}

In this section we discuss, in a general setting, the classical
method of taking parametric derivatives to obtain solutions of
differential equations 
at special values of the parameters. 
This technique gives a partial understanding of the parametric variation of solutions of $H_A(\beta)$, 
as shown in Corollary~\ref{cor:sheaf}, which is a weaker version of Theorems~\ref{thm:main}.

Let $P_1(x,\del),\dots,P_m(x,\del)$ be linear partial differential
operators in $D[\beta]$, the Weyl algebra on $\CC^n$ with additional commuting variables $\beta = \beta_1,\dots,\beta_d$. 
Let $\varphi_1(\beta),\dots,\varphi_m(\beta)$ be polynomial functions of
$\beta \in \CC^d$, viewed as elements of $D[\beta]$. 
Consider the system $H(\beta)$ of parametric linear partial differential equations given by
\begin{equation}
\label{eqn:ParametricDerivativesPDE}
P_i(x,\del) \bullet \Phi(x,\beta) = \varphi_i(\beta) \Phi(x, \beta),
\qquad i=1,\dots, m.
\end{equation}
We view $H(\beta)$ as a left ideal in $D[\beta]$ defined by the operators
$P_i(x,\del) - \varphi_i(\beta)$.   

Assume further that $\Phi(x, \beta)$ is locally analytic for $x$ in an open set $V\subseteq\CC^n$ and 
$\beta$ in a neighborhood $\mathcal{U}$ (in the analytic topology) of $\bar{\beta} \in \CC^d$. 
For $\gamma \in \CC^d$, denote by $\nabla_\gamma$ the 
differentiation operator with respect to $\beta$ in the direction of
$\gamma$. 
Applying $\nabla_\gamma$ to both sides of~\eqref{eqn:ParametricDerivativesPDE} yields
\[
P_i(x,\del)\bullet \nabla_\gamma \Phi(x,\beta) = \nabla_\gamma\varphi_i(\beta)
\cdot \Phi(x,\beta) + \varphi_i(\beta) \cdot \nabla_\gamma \Phi(x,\beta),
\qquad i=1,\dots, m.
\]
Thus, if $\bar{\beta}$ is such that $\Phi(x,
\bar{\beta}) \equiv 0$ for $x\in V$, then it follows that $\nabla_\gamma \Phi(x,\bar{\beta})$
solves~\eqref{eqn:ParametricDerivativesPDE} at $\beta = \bar{\beta}$. 
In this case, we say
that $\nabla_\gamma \Phi(x,\bar{\beta})$ has been \emph{constructed from $\Phi(x,\bar{\beta})$ by taking parametric derivatives}. 

If the function $\nabla_\gamma \Phi(x, \beta)$ happens to vanish
identically at $\beta = \bar{\beta}$, then the parametric derivative procedure can be
iterated. 
We now state a sufficient condition for this algorithm to terminate after a finite number of
steps. 
Note that the extended Euler--Mellin integrals provide a basis of solutions for $H_A(\beta)$ at nonsingular $x$ and nonresonant $\beta$ and thus satisfy the hypotheses of this result. 

\begin{proposition}
\label{pro:ParametricDerivatives}
Let $\Phi(x, \beta)$ be a solution
to~\eqref{eqn:ParametricDerivativesPDE}, which is analytic in
$x$ in an open subset $V \subseteq \CC^n$, and locally analytic and not identically vanishing in a neighborhood $\cU$ of $\bar{\beta}$. Assume further
that $\Phi(x, \beta)$ is identically vanishing when $\beta$ is restricted to a
hyperplane $L \subset \CC^d$ containing $\bar{\beta}$. Then for each $\gamma\notin T_L(\bar{\beta})$, the tangent space of $L$ at $\bar{\beta}$, the above process terminates after a finite number of steps. That is, for each $\bar{\beta}$, there is an integer $q = q(\bar{\beta},\gamma)$ such that $\nabla_\gamma^{(q)}\Phi(x, \beta) \not\equiv 0$ for $\beta \in L\cap \cU$. 
Furthermore, the number $q$ does not depend on $\gamma$, and thus it defines
a function $\bar{\beta} \mapsto q(\bar{\beta})$, 
which in turn is upper semicontinuous 
in the analytic topology of $\CC^d$.
\end{proposition}
\begin{proof}
Let $v_1, \dots, v_{d-1} \in T_L(\bar{\beta})$ extend by $\gamma$ to a basis of $\CC^d$. Then 
\[
\nabla_{v_i}^{(q)} \Phi(x, \beta) \equiv 0 \quad \text{for} \quad i \in\{ 1, \dots, d-1\}, \quad q\geq 0,\quad \text{and}\quad  \beta \in L\cap \cU.
\]
If in addition, $\nabla_\gamma^{(q)}\Phi(x, \beta) = 0$ for all $q$ and $\beta \in L\cap \cU$, then
all mixed derivatives $\nabla_{v_i}^{(q_i)} \nabla_\gamma^{(q)}\Phi(x, \beta)$ also vanish in $L\cap \cU$; in particular, they vanish at $\bar{\beta}$. Hence  Taylor's formula implies that
$\Phi(x, \beta) = 0$ for all $\beta$ in a neighborhood of
$\bar{\beta}$, a contradiction. That 
$q$ does not depend on $\gamma$ follows by a change of variables fixing
$v_1, \dots, v_{d-1}$. To see that the map $\bar{\beta} \mapsto
q(\bar{\beta})$ is upper semicontinous, 
it is enough to note that $\nabla_\gamma^{(q)}\Phi(x, \beta) \not \equiv
0$ for $\beta \in L\cap \cU$ implies that $\nabla_\gamma^{(q)}\Phi(x, \beta) \not \equiv 0$ for
$\beta$ in some (analytic) open neighborhood of $\bar{\beta}$. 
\end{proof}

\begin{proposition}
\label{pro:NoBadPoints}
With the hypotheses of Proposition \ref{pro:ParametricDerivatives}, assume further that in a neighborhood $\cV$ of $\bar \beta$, the function $\Phi(x, \beta)$ is not identically vanishing in $x$ for any fixed $\beta\in\cV\minus L$. 
Then $\nabla_\gamma^{(q)}\Phi(x, \beta)$ 
is not identically vanishing in $x$ at $\bar \beta$.
\end{proposition}
\begin{proof}
We can assume that $\gamma$ is the normal vector for the hyperplane $L$, so that $L$ is defined by a linear equation $\gamma\cdot\beta = \kappa$.
By l'H\^opital's rule, for $\beta\in\cV\cap L$, 
\[
\nabla_\gamma^{(q)}\Phi(x, \beta) = \left.\Phi(x, \beta)((\gamma\cdot\beta)-\kappa)^{-q}\right.
\]
Hence for $\beta\in\cV\cap L$, the function
$\nabla_\gamma^{(q)}\Phi(x,\beta)$ can be replaced  
by the analytic extension to $L$ of the ratio
\[
\Phi(x, \beta)((\gamma\cdot\beta)-\kappa)^{-q},
\]
which gives a solution of $H(\beta)$ that varies locally analytically in a neighborhood of $L$.

The vanishing locus of $\Phi(x, \beta)((\gamma\cdot \beta)-\kappa)^{-q}$ is an analytic subvariety of codimension at least $1$ in $L$. Thus, by assumption, this vanishing locus is of codimension at least $2$ in $\CC^d$, which implies that it is empty.
\end{proof}

For any $\beta \in \CC^d$, the \emph{characteristic variety} of $H(\beta)$ is
\vspace{-1mm}
\begin{equation}
\charVar(H(\beta)) \defeq \Var(\gr^F(H(\beta))) \subseteq T^*\CC^n\cong\CC^{2n}, 
\end{equation}
where $F$ is the order filtration on $D$, given by the order of differential operators.
The \emph{singular locus} of $H(\beta)$, denoted $\Sing(H(\beta))$, 
is the image under the projection $T^*\CC^n\to\CC^n$ given by $(x,\xi)\mapsto x$ of $\charVar(H(\beta))
\minus \Var(\xi_1,\dots,\xi_n)$. 
Given a generic nonsingular point $\bar{x}
\in \CC^n$ for $H(\beta)$, denote by $\mathcal{O}^{\rm an}_{\CC^n,\bar{x}}$ the 
space of multivalued germs of analytic functions at $\bar{x}$. 
The $\CC$-vector space 
\[
\Sol_{\bar{x}}(H(\beta)) \defeq 
\{ \Phi\in\mathcal{O}^{{\rm an}}_{\CC^n,\bar{x}} 
\mid Q\bullet\Phi = 0\text{ for all } Q\in H(\beta)\}
\]
is the \emph{solution space of $H(\beta)$ at $\bar{x}$}. 

We introduce one more notion before stating the main result of this section. 
Given a countable, locally finite union $\sL$ of
linear varieties in $\CC^d$, let $\cS(\sL)$ denote the natural stratification of $\CC^d$ induced by the flats of $\sL$. In other words, the closure of a stratum of $\cS(\sL)$ is an intersection of irreducible components of $\sL$, and $\cS(\sL)$ is the coarsest possible stratification of $\CC^d$ with this property.

\begin{theorem}
\label{thm:ParDerSolutionSheaf}
Consider a system $H(\beta)\subseteq D[\beta]$ of linear partial differential 
equations 
of the form~\eqref{eqn:ParametricDerivativesPDE}. 
Suppose that there exists a countable, locally finite union $\sL$ of
linear varieties in $\CC^d$ and an open set $V\subseteq\CC^n$,
disjoint from $\Sing(H(\beta))$ for all $\beta$, such that 
\[
\dim_{\CC}({\rm Sol}_{\bar{x}}(H(\beta))) = r < \infty\text{ for all }\beta\notin \sL \text{ and } x\in V.
\] 
Suppose further that there exists a set of functions that are locally analytic
and form a basis of $\Sol_x(H(\beta))$ for $(x,\beta)\in V\times(\CC^d\minus \sL)$. 
Then along each stratum $S\in\cS(\sL)$, and for any $x\in V$, there is an 
$r$-dimensional subspace of $\Sol_x(H(\beta))$ which, when restricted to 
$S$, is locally analytic in $\beta$.
\end{theorem}
\begin{proof}
Let $\Phi_1(x,\beta), \dots, \Phi_r(x,\beta)$ be a collection of solutions
that are locally analytic and form a basis of $\Sol_x(H(\beta))$
for $(x,\beta)\in V\times (\CC^d\setminus \sL)$.

We first claim that it is enough to consider the 
case that $\sL$ is a union of hyperplanes. 
To see this, let  $K \subset \sL$ be an irreducible 
component of $\sL$ of codimension at least two, and 
let $K_0$ denote the unique maximal stratum of 
$\cS(\sL)$ that is contained in $K$. 
Since $\sL$ is locally finite,
$K_0$ is a nonempty Zariski open subset of $K$.
Note that for any $\beta\in K_0$, there exists an 
analytic neighborhood $N$ of $\beta$ such that 
$N\cap K_0 = N\cap K$ and furthermore,
any linear combination of the functions 
$\Phi_1, \dots, \Phi_r$ is analytic and
nonvanishing for $\beta\in N\minus K$. 
Since $K$ is of codimension at least two, 
we conclude that any linear combination of
$\Phi_1,\dots, \Phi_r$ has a unique nonvanishing analytic continuation
to $N$. 

We may thus assume that $\sL$ is a union of hyperplanes. To proceed with the proof, 
we use induction on $d$. The base case $d = 0$ is trivial. 
Thus, assume now that the statement is proven for $d < k$, and consider the case $d = k$.

Proposition~\ref{pro:ParametricDerivatives} can be applied to a collection 
of linearly independent solutions of the system $H(\beta)$ for all 
$\beta\in\CC^k\minus \sL$, as follows. 
If there is a linear combination
\[
\Phi(x,\beta) = \sum_{i=1}^r c_i \Phi_i(x, \beta)
\]
that encodes a linear dependence relation valid when $\beta \in L$, 
for some linear space $L\subseteq \sL$, then Proposition~\ref{pro:ParametricDerivatives} can be applied to the solution $\Phi(x, \beta)$. 
Further, if there are several linear dependencies that are valid along $L$, then Proposition~\ref{pro:ParametricDerivatives} can be applied repeatedly. 

Consider a hyperplane $L \subset \sL$.
Let ${\sL}_L \subset L$ denote the set
\[
\sL_L = \bigcup_{\substack{S \in \cS(\sL)\\ \overline{S}
\subsetneq L}} S.
\]
Note that $\sL_L$ is a countable, locally finite union of hyperplanes in 
$L \cong \CC^{d-1}$.
For a fixed $\bar{\beta}$ contained $L\minus \sL_L$, 
Proposition~\ref{pro:NoBadPoints} shows that the solutions obtained 
using Proposition~\ref{pro:ParametricDerivatives} do not vanish identically 
in $x$ at $\bar{\beta}$. Thus, by Proposition~\ref{pro:ParametricDerivatives}, 
we obtain $r$ locally analytic functions which are nontrivial and 
linearly independent for $\beta \in L\minus \sL_L$. 
Furthermore, the polynomial functions
$\varphi_i(\beta)$, $i=1, \dots, m$,
are still algebraic when restricted to the 
hyperplane $L$. In particular, by induction, 
for each stratum $S \in \cS(\sL_L)$, there is a set of $r$ linearly independent elements of $\Sol_x(H(\beta))$ for each $\beta \in S$, which are locally analytic. 

Consider a stratum $S\in \cS(\sL)$ of codimension $c>1$. 
Such a stratum is the intersection of $L_1,\dots,L_c$, which are hyperplanes in $\sL$. 
By the previous paragraph, along each $L_i$, there is an $r$-dimensional solution space of $H(\beta)$ at $x\in V$ for all $\beta\in L_i$. 
In general, these solution spaces need not coincide along $S$; 
however, to complete this proof, it
suffices to choose one of the available solution spaces for each such 
stratum $S$.
\end{proof}

\begin{theorem}
\label{thm:SolutionsAnalyticStratification}
Assume that the $A$-hypergeometric system $H_A(\beta)$ is such that the set of rank-jumping parameters $\cE_A$ is at most zero-dimensional. 
Then there is a stratification of $\CC^d$ such that for each stratum $S$, the solutions of $H_A(\beta)$ are locally analytic as functions of $x$ and $\beta$ in $(\CC^n \minus \Sigma_A)\times S$.
\end{theorem}
\begin{proof}
At nonresonant parameters $\beta$, there is a basis of the solution space of $H_A(\beta)$ given by Euler-type integrals over compact cycles~\cite[Theorem~2.10]{GKZ90}. 
These compact cycles can be chosen uniformly for all parameters, so the corresponding integrals yield entire functions of the parameter $\beta$ (cf. Section~\ref{sec:EM:genericResonant}). 
Applying Theorem~\ref{thm:ParDerSolutionSheaf} yields for each stratum $S\in\cS(\ResArr(A))$, a set of $\vol(A)$-many linearly independent solutions of $H_A(\beta)$ that are locally analytic
for $(x,\beta)\in(\CC^n \minus \Sigma_A)\times S$. 

The proof is now complete if the set of rank-jumping parameters $\cE_A$ is empty. 
If not, refine $\cS(\ResArr(A))$ by placing each rank-jumping parameter in its own stratum. 
Then at each $\beta\in\cE_A$, complete a basis of elements of $\Sol_x(H_A(\beta))$ for $x\in\CC^n \minus \Sigma_A$, as there is nothing to prove regarding analyticity in $\beta$ for such parameters because $\cE_A$ is zero-dimensional.
\end{proof}

\begin{corollary}
\label{cor:sheafCM}
If $\CC[\del]/I_A$ is Cohen--Macaulay, then for any stratum $S\in\cS(\ResArr(A))$, the solutions of $H_A(\beta)$ are locally analytic as functions of $x$ and $\beta$ in $(\CC^n\minus\Sigma_A)\times S$. 
\end{corollary}
\begin{proof}
When $\CC[\del]/I_A$ is Cohen--Macaulay, $H_A(\beta)$ has no rank jumping parameters~\cite[Corollary~9.2]{MMW}. Thus the result follows immediately from Theorem~\ref{thm:SolutionsAnalyticStratification}. 
\end{proof}

Returning to the case of projective toric curves, the following is another consequence of Theorem~\ref{thm:SolutionsAnalyticStratification}.

\begin{corollary}
\label{cor:sheaf}
Assume that $A$ has the form~\eqref{eqn:A}. Let $\mathcal{C}$ be the
subset of $\ResArr(A)$ consisting of points lying on two resonant
lines, and form a decomposition $\CC^2 = (\CC^2 \minus \ResArr(A) )
\cup (\ResArr(A) \minus \mathcal{C} ) \cup \mathcal{C}$.
Then the solutions of $H_A(\beta)$ are locally analytic as functions
of $x$ and $\beta$ on $(\CC^n\minus \Sigma_A) \times (\CC^2 \minus
\ResArr(A))$ and also on $(\CC^n\minus \Sigma_A) \times (\ResArr(A)
\minus \mathcal{C})$.
\end{corollary}

In Corollary~\ref{cor:sheaf}, note that it might be possible to take a strictly smaller union of lines, as we show in Theorem~\ref{thm:bigSheaf}. 

\begin{remark}
\label{rmk:higherDimensions}
If $\cE_A$ is not zero-dimensional, 
then Theorem~\ref{thm:ParDerSolutionSheaf} 
cannot be used to explain how to create (and control with respect to $\beta$)
additional linearly independent solutions along higher-dimensional components of $\cE_A$. 
Further tools are being developed to address this issue.
\endrk 
\end{remark}

Although Corollary~\ref{cor:sheaf} provides information on the
parametric behavior of the solutions of a hypergeometric system
associated to a projective toric curve, this information is not very
explicit, and in particular, does not involve the formation of rank jumps. 
In the upcoming sections, we work towards the significantly stronger
Theorems~\ref{thm:main} and~\ref{thm:bigSheaf}.

\section{Resonant parameters and Euler--Mellin integrals}
\label{sec:EM:genericResonant}

In this section, we study the behavior of extended Euler--Mellin integrals at parameters that are resonant with respect to precisely one facet of $A$, see Definition~\ref{def:resonant}. 
We obtain an explicit understanding of how these solutions behave along a line in $\Pol(A)$,
while the behavior along nonpolar lines seems to be fundamentally different.

We first consider the case that $\beta$ is polar, where the behavior is illustrated by the following example.
A series solution $\sum_{v\in \CC^n} c_v x^v$ of $H_A(\beta)$ is said to have \emph{finite support} if the set $\{v\in \CC^n\mid c_v\neq 0\}$ is finite. 

\begin{example}
\label{ex:SupportingHyperplane}
For any $A$ as in~\eqref{eqn:A}, consider the supporting 
line of the cone~\eqref{eqn:ConvergenceDomain} given by 
$\beta_2=0$, which is contained in $\Pol(A)$. In order to 
evaluate $\Psi_f^\Theta$, we must expand $M_f^\Theta(x, \beta)$ once 
in the direction of $\mu_0$, which yields 
\[
M_f^\Theta(x,\beta)
= \frac{\beta_1}{\beta_2} \int_{\Arg^{-1}(\theta)} z^{-\beta_2+1} \frac{f'(z)}{f(z)^{-\beta_1+1}}\,\frac{dz}{z},
\]
for $\theta\in\Theta$. 
Hence, along the line $\beta_2 = 0$, the Euler--Mellin integral for $\Theta$ evaluates, after applying an extension formula once, as follows:
\[
\left.\frac{M_f^\Theta(x,\beta)}{\Gamma(\beta_2)}\right|_{\beta_2 = 0} 
= \beta_1 \int_{\Arg^{-1}(\theta)} \frac{f'(z)}{f(z)^{-\beta_1+1}}\, dz 
= \beta_1 f(0)^{\beta_1} = \beta_1 x_1^{\beta_1}.
\]
Most notably, $\Psi_f^\Theta(\beta_1,0)$ is independent of $\Theta$ and equal to a series solution of $H_A(\beta)$ with finite support.
\endrk
\end{example}

\begin{theorem}
\label{thm:genericResPolar}
If $L$ is a line contained in $\Pol(A)$, then for all $\beta\in L$, all extended Euler--Mellin integrals $\Psi_f^\Theta(x, \beta)$ coincide and evaluate to a 
finitely supported series solution of $H_A(\beta)$. 
After possibly removing a factor which is constant with respect to $x$, this series is nonvanishing in $x$ for each $\beta$ and
varies locally analytically on $(\CC^n\minus\Var(x_1x_n))\times L$. 
\end{theorem}
\begin{proof}
It is enough to consider $x\in V$ as in Theorem~\ref{thm:EMindep}. 
We first consider the case $\beta_2 = N$, for some integer $N$, so that $L = L_0(N)$ is an integer translate of the span of the face $\left\{\fface\right\}$. 
From~\eqref{eq:extendedEM}, since $L\subseteq\Pol(A)$, there exists a partition of $N$ using only $k_2, \dots, k_n$. 
Let $P(A,N)$ denote the set of all ordered partitions $p$ of $N$ with parts $k_2, \dots, k_n$. 
Let $m_i(p)$ denote the number of times $k_i$ appears in $p$, and set $m(p)\defeq(m_1(p), \dots, m_n(p))$, so that $|m(p)|$ is the length of $p$. 
Furthermore, for each $i\in \{1,\dots, |m(p)|\}$, let $s_i \defeq s_i(p)$ denote the $i$th partial sum of $p$. 
Now parameterize the line 
$L_0(N) = \{\beta\in\CC^2 \mid \beta_2 = N\}$ by
\[
\lambda \mapsto (\lambda, N) \quad\text{for}\quad \lambda \in \CC.
\] 
We claim that, for each connected component $\Theta$ of the complement of $\sC(f)$, the Euler--Mellin integral for $\Theta$ evaluates, after applying the extension formulas, to
 \begin{equation}
 \label{eqn:SolutionAlongPolar}
  \left.\frac{M_f^\Theta(x,\beta)}
  {\Gamma(\beta_2)}\right|_{L_0(N)}  
  = x_1^{\lambda} \sum_{p\in P(A,N)}
  \left(\prod_{i=1}^{|m(p)|-1} \frac{\lambda - i}{N-s_i(p)}\right)
  \left(\prod_{i=2}^{n}k_i^{m_i(p)}\left(\frac{x_i}{x_1}\right)^{m_i(p)}\right).
 \end{equation}
Indeed, the factor $\Gamma(\beta_2)$ in~\eqref{eq:extendedEM} implies that the only terms in
the expansion of $M_f^\Theta(x,\beta)$ that are relevant at $\beta_2 =
N$ are those containing the factor $\beta_2 - N$ in the denominators of
their coefficients. 
This is the case only for terms corresponding to ordered partitions of $N$, and the term corresponding to $p\in P(A,N)$, when restricted to $L_0(N)$, is the one given in~\eqref{eqn:SolutionAlongPolar}, including the monomial $x^{m(p)}$. 
The integral of this term evaluates, in the manner of Example~\ref{ex:SupportingHyperplane}, to $x_1^{\beta_1-|m(p)|}$. 

By symmetry, a similar formula can be given 
when 
$k\beta_1-\beta_2 = N$, so that $L = L_k(N)$ is a translate of the span of the face 
$\left\{\lface\right\}$. In this case, consider the parameterization of $L_k(N)$ given by 
\[
\lambda \mapsto (\lambda, k\lambda + N) \quad\text{for}\quad \lambda \in \CC. 
\]
The analog of~\eqref{eqn:SolutionAlongPolar} for this line is 
\begin{equation}
\label{eqn:SolutionAlongPolar2}
  \left.\frac{M_f^\Theta(x,\beta)}
  {\Gamma(k\beta_1-\beta_2)}\right|_{L_k(N)}  
= x_n^{\lambda} \sum_{p\in  P(\bar A,N)}
  \left(\prod_{i=1}^{|m(p)|-1}
  \frac{\lambda - i}{N-s_i(p)}\right)\, 
 \left(\prod_{i=1}^{n-1} (k-k_i)^{m_i(p)} \left(\frac{x_i}{x_n}\right)^{m_i(p)}\right).
\end{equation}

Now the coefficient of a monomial $x_1^{\lambda - |m_p|} \prod_{i=2}^n x_i^{m_i(p)}$ in \eqref{eqn:SolutionAlongPolar} or $x_n^{\lambda - |m_p|} \prod_{i=2}^{n-1} x_i^{m_i(p)}$ in \eqref{eqn:SolutionAlongPolar2} is a sum of positive multiples of the product
\begin{equation}
\label{eqn:CoeffProduct}
\prod_{i = 1}^{|m(p)|-1}(\lambda - i), 
\end{equation}
and hence it is a positive multiple of \eqref{eqn:CoeffProduct}; in particular, the coefficient of a monomial vanishes only if
the product \eqref{eqn:CoeffProduct} vanishes.
Thus a function of the form~\eqref{eqn:SolutionAlongPolar} or~\eqref{eqn:SolutionAlongPolar2} is nontrivial unless $\lambda$ is such that all the coefficients vanish, which is equivalent to $\lambda$ being a positive integer such that every $p\in P(A,N)$ has more than $\lambda$ terms. However, this vanishing can be removed, as it is caused by a factor $(\lambda - i)$ appearing in all coefficients
of \eqref{eqn:SolutionAlongPolar} or \eqref{eqn:SolutionAlongPolar2}. 
Removing such factors,~\eqref{eqn:SolutionAlongPolar} and \eqref{eqn:SolutionAlongPolar2} provide everywhere nontrivial solutions that vary locally analytically in $\lambda$ along $L_0(N)$ or $L_k(N)$, respectively. Local analyticity on $\CC^n\minus\Var(x_1x_n)$ is clear from the denominators of~\eqref{eqn:SolutionAlongPolar} and~\eqref{eqn:SolutionAlongPolar2}.
\end{proof}

The behavior of extended Euler--Mellin integrals along nonpolar resonant lines is more difficult to track. 
To illustrate the difference, we turn to a discussion of the behavior of residue integrals. 

Still using $f(z) = \sum_{i=1}^{n} x_i z^{k_i}$, consider an integral of the form
\begin{equation}
\label{eqn:HypergeometricIntegral}
\int_C {f(z)^{\beta_1}}{z^{-\beta_2}}\,
\,\frac{dz}{z},
\end{equation}
where $C$ is some cycle in $\CC\minus\Var(f)$. It is shown in~\cite[Theorem~2.7]{GKZ90} that~\eqref{eqn:HypergeometricIntegral} is $A$-hypergeometric as a germ in $x$, provided that it has sufficiently nice convergence properties. This is the case when $C$ is compact. 
The integrand in~\eqref{eqn:HypergeometricIntegral} has, for very
general $\beta$ (including all nonresonant parameters),  singularities
at $0$, $\infty$, and at each $\rho \in \Var(f)$. This yields $\vol(A) + 2$ residue integrals
\[
\Res_\rho(x,\beta) \defeq \int_{C(\rho)} {f(z)^{\beta_1}}{z^{-\beta_2}}\,\,\frac{dz}{z}, 
\]
where $C(\rho)$ is a small counterclockwise loop around $\rho$ if
$\rho\in\CC$ equals zero or a root of $f$, or it is a large clockwise
oriented loop if $\rho = \infty$. 
The integral 
\eqref{eqn:HypergeometricIntegral} is multivalued
in two senses. It is multivalued in $x$, as an 
analytic continuation along a loop in $(\CC^*)^n$
need not preserve the cycle $C$; also, it depends on the choices of branches
of the exponential functions $f(z)^{\beta_1}$ and $z^{-\beta_2}$.
Note that the values of  \eqref{eqn:HypergeometricIntegral}
for different branches of these exponential functions 
differ only by multiplication times an
exponential function in the parameters. In particular, 
for a fixed cycle $C$, the $\CC$-vector space spanned 
by the integral in question is uniquely determined.
In the case when its cycle of integration $C$ is compact, 
\eqref{eqn:HypergeometricIntegral} is entire in $\beta$, as its integrand shares this property. 

\begin{convention}
\label{conv:zerosf}
Let $\rho_1, \dots, \rho_{k}$ denote the zeros of $f$, with indices considered modulo $k = \vol(A)$, ordered by their arguments. 
Similarly, label the components of the complement of $\sC(f)$ by $\Theta_i$, where the indexing is chosen so that
the sector $\Arg^{-1}(\Theta_i)$ contains $\rho_{i-1}$ and $\rho_{i}$ in its boundary.
\end{convention}

Because of the multivaluedness of \eqref{eqn:HypergeometricIntegral}, at generic $\beta$, the vanishing
theorem of residues does not apply.
However, for parameters $\beta$ contained in~\eqref{eqn:ConvergenceDomain}, it follows from the estimates in the proof of~\cite[Theorem 2.3]{BFP} that for each $\rho_i \in \Var(f)$, 
\begin{equation}
\label{eqn:resEM}
\Res_{\rho_i}(x,\beta) = M_f^{\Theta_i}(x,\beta) - M_f^{\Theta_{i+1}}(x,\beta),
\end{equation}
provided that the branches of 
(the integrands of) 
the Euler--Mellin integrals in the right hand side are chosen 
consistently.
Note that by the uniqueness of meromorphic extension, this 
equality holds for all nonresonant $\beta$, when on the right 
hand side, each $M_f^{\Theta_j}(x,\beta)$ is replaced by 
$\Psi_f^{\Theta_j}(x, \beta)$.

It is natural to decompose the set of residue integrals into two types, where the first
is given by the residues at the roots $\rho \in \Var(f)$ and the second is
given by the residues at $0$ and $\infty$. 
As we now show, along polar lines, the first type of residue integrals vanish; however, along nonpolar resonant lines, it is
instead the residue integrals of the second type that vanish. 

\pagebreak
\begin{proposition} \ 
\label{prop:genericResNonpolar:Residues}
\vspace{-3mm}
\begin{enumerate}
\item If $\beta \in \Pol(A)$, then $\Res_{\rho}(x, \beta) = 0$ for each $\rho\in \Var(f)$.
\item If $\beta\in\ResArr(A)\minus\Pol(A)$, 
then at least one of $\Res_{0}(x,\beta)$ and $\Res_{\infty}(x,\beta)$
vanishes identically in $x$.
\end{enumerate}
\end{proposition}
\begin{proof}
Note that (1) follows from~\eqref{eqn:resEM} and Theorem~\ref{thm:genericResPolar}.

For (2), it is enough to consider the case $\beta_2 = N$ for some $N\in \ZZ$, so that $\beta$ is resonant with respect to the face $\left\{\fface\right\}$.  (If $\beta$ is instead resonant with respect to the face
$\left\{\lface\right\}$, then apply the change of variables given by $z\mapsto 1/z$ to return to this case.) 

We claim that $\Res_{0}(x,\beta)$ vanishes identically in $x$.
If $N$ is a nonpositive
integer, then the integrand of~\eqref{eq:EMintegral} is
analytic in a neighborhood of $z = 0$. Thus, the only case to
consider is when $N$ is positive, which corresponds to a gap in the set $\Pol(A)$ of polar lines.

Since residue integrals involve compact cycles, they converge for every $\beta$.
However, by uniqueness of meromorphic extension, the
identities~\eqref{eqn:ExtensionFormulas-1}
and~\eqref{eqn:ExtensionFormulas-2} hold. 
Hence 
\[
\Res_{0}(x,\beta) 
= \frac{\beta_1}{\beta_2} \int_{C(0)} z^{-\beta_2+1} f'(z)f(z)^{\beta_1-1}\,\frac{dz}{z},
\]
where $f'(z)$ denotes the derivative of $f(z)$ with respect to $z$. 
By applying this formula repeatedly, 
we conclude that it is enough to show that each term of $f'(z)$ which corresponds to $k_i > \beta_1$ vanishes. Indeed, this term equals
\[
\frac{\beta_1}{\beta_2} \int_{C(0)} z^{-\beta_2+1} k_i z_i^{k_i-1}f(z)^{\beta_1-1}\,\frac{dz}{z},
\]
whose integrand has a zero at the origin of order $k_i - \beta_1 - 1 \geq 0$. 
\end{proof}

\begin{remark}
\label{rmk:linIndep?}
It is not clear if any further dependencies of the residue integrals
exist at nonpolar resonant parameters. 
Thus, the implications of
Proposition~\ref{prop:genericResNonpolar:Residues}.(2) for extended
Euler--Mellin integrals also remain unclear. Computational evidence
suggests that at nonpolar resonant parameters, the analytic
continuations of the extended Euler--Mellin integrals remain linearly
independent. If this statement were proven, then the decomposition of
$\CC^2$ in Theorem~\ref{thm:main} could instead be given as
$(\CC^2\minus\Pol(A))\cup\Pol(A)$. On the other hand, all the lines in
the arrangement $\sL$ from Theorem~\ref{thm:bigSheaf} are
polar, a fact that is proved using hypergeometric series.
\endrk
\end{remark}

\section{Intersections of polar lines and Euler--Mellin integrals}
\label{sec:EM:nongenericResonant}

In this section, we explain the behavior of extended Euler--Mellin integrals at parameters $\beta$ that lie at the intersection of two resonant lines, so that $\beta$ is resonant with respect to both faces of $A$, see Definition~\ref{def:resonant}. 
The partition of the resonant arrangement into polar and nonpolar lines gives a natural partition of these parameters into three distinct cases.

At intersections of two nonpolar lines, and also at intersections of one polar and one nonpolar line, the rank of $H_A(\beta)$ is $\vol(A) = k$. 
As discussed in Remark~\ref{rmk:linIndep?}, it not clear how (or if) the extended Euler--Mellin integrals form linear dependencies at such parameters. 
We resolve this issue later through the use of series solutions of $H_A(\beta)$ in Theorem~\ref{thm:seriesLinIndep}. 

Thus for now, we consider only the third case, when a parameter $\beta$ is contained in the intersection of two polar lines, which, in particular, includes all rank-jumping parameters. 
Associated to these two lines, respectively, are the two finitely supported series as described in Theorem~\ref{thm:genericResPolar}. 
Before we (possibly) remove unnecessary (vanishing) coefficients at the end of the proof of 
Theorem~\ref{thm:genericResPolar}, 
these two series are evaluations of the same integral and therefore coincide. 
However, after such coefficients are removed, the resulting two series may or may not coincide. 
For $\beta\in\ZZ^2$, this is decided by whether or not $\beta$ is rank-jumping.

Recall that $\cE_A \defeq \{\beta\in\CC^2\mid \rank\, H_A(\beta)>\vol(A)\}$ denotes the set of rank-jumping parameters. 

\begin{theorem}
\label{thm:nongenericResPolarEM}
Suppose that $\beta$ is at the intersection of two resonant lines, both of which are contained in $\Pol(A)$. 
\begin{enumerate}
\vspace{-2.5mm}
\item \label{item:prop:nongenericResPolar:nonZZ}
If $\beta\notin\ZZ^2$ or $\beta\in\cE_A$, then 
the two solutions 
obtained in Theorem~\ref{thm:genericResPolar} at $\beta$ 
are finitely supported and linearly independent.
\item \label{item:prop:nongenericResPolar:ZZ}
If $\beta\in\ZZ^2\minus\cE_A$, then for $x\in(\CC^n\minus\Sigma_A)$, 
then the two solutions 
obtained in Theorem~\ref{thm:genericResPolar} at $\beta$ 
coincide and equal a single finitely supported series.
\end{enumerate}
\end{theorem}

If $\beta \in \NN A$,~\cite[Lemma~3.4.10]{SST}
(or~\cite[Proposition~1.1]{CDD}) show that $H_A(\beta)$ has a unique
(up to multiplication by a constant) polynomial solution. By
Theorem~\ref{thm:nongenericResPolarEM}, this solution is given by an
extended Euler--Mellin integral.

\begin{example}
\label{ex:0134:finiteSupp}
We compute the finitely supported solutions provided by
Theorem~\ref{thm:genericResPolar} when    
$A = \left[\begin{smallmatrix}1&1&1&1\\0&1&3&4\end{smallmatrix}\right]$. 
At $\beta = \left[\begin{smallmatrix}1\\ 1\end{smallmatrix}\right]\in\NN A$,~\eqref{eqn:SolutionAlongPolar} and~\eqref{eqn:SolutionAlongPolar2} both produce $x_1$. 
Next, at $\beta = \left[\begin{smallmatrix}1/2\\ 1\end{smallmatrix}\right]$,~they yield the
two finitely supported Puiseux series $x_1^{-\frac 1 2} x_2$ and $x_3 x_4^{-\frac 1 2}$. 
Finally, at the unique rank-jumping parameter for this $A$, $\beta = \left[\begin{smallmatrix}1\\ 2\end{smallmatrix}\right]$, they provide the two Laurent monomials $x_1^{-1} x_2^2$ and $x_3^2 x_4^{-1}$.
\endrk
\end{example} 

The \emph{support} of a series solution $\sum_{v\in \CC^n} c_v x^v$ of $H_A(\beta)$ is the set $\{v\in \CC^n\mid c_v\neq 0\}$, and the \emph{negative support} of a monomial $c_v x^v$ is the set $\{i\mid v_i\in\ZZ_{<0}\}$. 

\begin{proof}[Proof of Theorem~\ref{thm:nongenericResPolarEM}]
If $\beta\notin\ZZ^2$, then the series solutions~\eqref{eqn:SolutionAlongPolar}
and~\eqref{eqn:SolutionAlongPolar2} of $H_A(\beta)$ have disjoint supports. 
In particular, they are linearly independent, and the only function
which is simultaneously a constant multiple of both is the trivial
(identically zero) function.  

To handle the case $\beta\in\ZZ^2$, 
consider the matrix $\Delta \defeq CA$, where
$C = \left[\begin{smallmatrix} k&-1\\ 0 &\phantom{-}1\end{smallmatrix}\right]$. 
Since $C\in\mathrm{GL}_2(\QQ)$, it induces an isomorphism of $A$-hypergeometric systems  for any $\beta\in\CC^2$: 
\[
H_A(\beta) \cong H_\Delta(\gamma) \quad\text{where $\gamma \defeq C\beta$}. 
\]
Figure~\ref{fig:AVsDelta} illustrates the change of coordinates induced by $C$ on the parameter space. 
Let $\NN\Delta$ denote the monoid spanned by the columns of $\Delta$. 

\begin{figure}[t]
\centering
\includegraphics[angle=180,width=50mm]{023v4}
$\quad$
\includegraphics[angle=180,width=50mm]{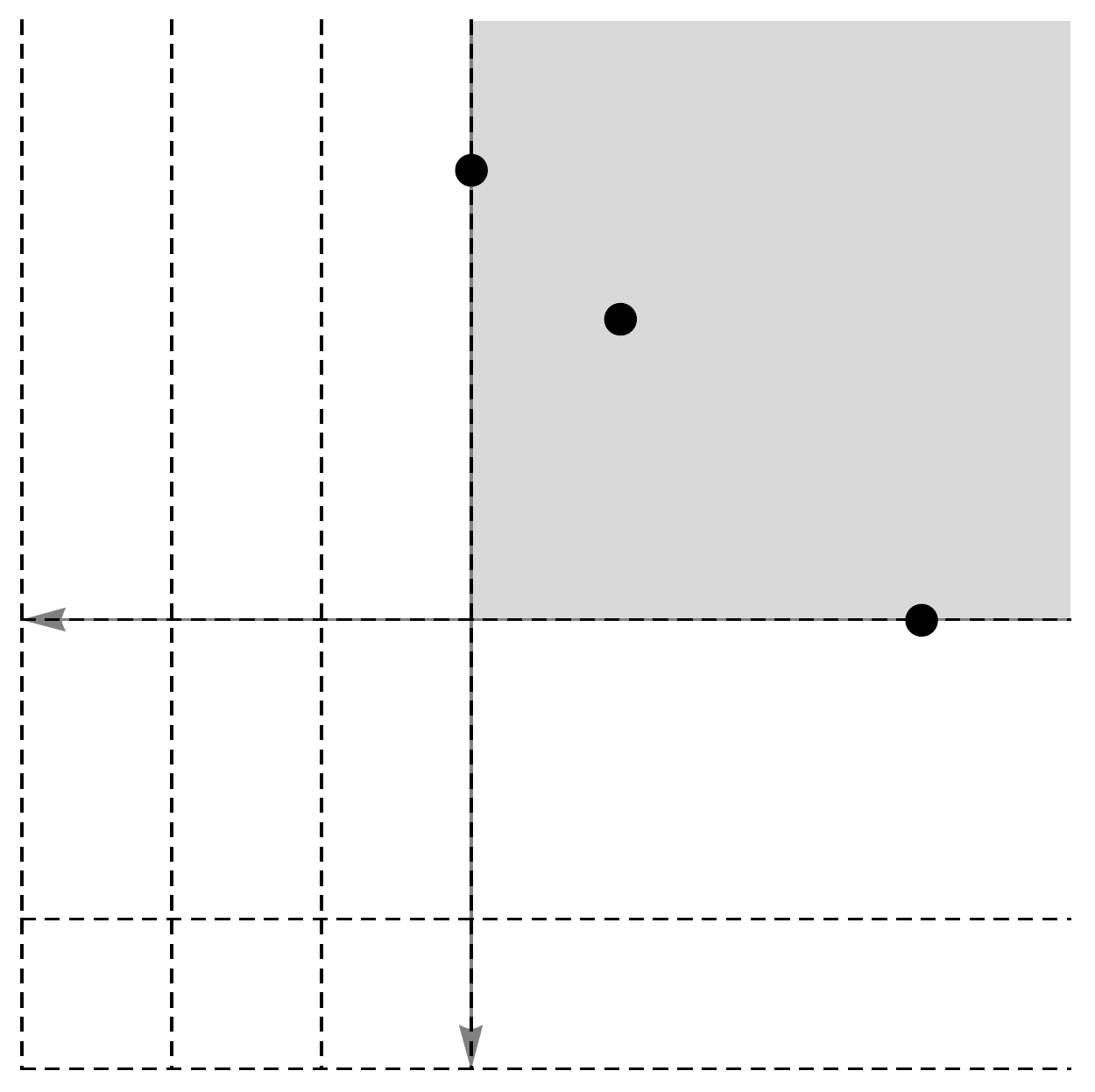}
\caption{
The left diagram shows $\Pol(A)$, while the right shows $\Pol(\Delta)$, mimicking Figure~\ref{fig:PolesVsResonance}.
}
\label{fig:AVsDelta}
\end{figure}

With this notation, when $\beta\in\ZZ^2$ is contained in two lines of $\Pol(A)$, consider all sets (if any) of the form 
$\{\delta^{i,j}\in\NN\Delta \mid 1\leq i,j\leq 2\}$ such that, with $\gamma = C\beta$, all of the following hold: 
\begin{align}
\label{eq:gamma1}
&\delta^{1,1}_1= \gamma_1, \quad
\delta^{1,1}_2+\delta^{1,2}_2 = \gamma_2, \quad
 \delta^{1,2}\neq \left[\begin{smallmatrix}0\\0\end{smallmatrix}\right],\\ 
\label{eq:gamma2}
&\delta^{2,1}_2= \gamma_2, \quad
\delta^{2,1}_1+\delta^{2,2}_1 = \gamma_1, \quad
\delta^{2,2}\neq \left[\begin{smallmatrix}0\\0\end{smallmatrix}\right].  
\end{align}
The conditions~\eqref{eq:gamma1} and~\eqref{eq:gamma2} are necessary for the 
equations of the polar lines containing $\beta$ 
to appear in the coefficient of some term after 
the Euler--Mellin integrals have been expanded either first over the face $\{\fface\}$, or first
over the face $\{\lface\}$ respectively. 
From this, we immediately conclude that if there is no set of the form $\{\delta^{i,j}\in\NN\Delta\mid 1\leq i,j\leq 2\}$ such that~\eqref{eq:gamma1} and~\eqref{eq:gamma2} hold, then all extended Euler--Mellin integrals vanish.

Let us now show that 
when $\beta\in\ZZ^2$, all extended Euler--Mellin integrals vanish at $\beta$ if and only if $\beta$ is rank-jumping 
for $A$.
We first claim that, as $\beta\in\ZZ^2$ is contained in the intersection of two lines in $\Pol(A)$, 
there exists a pair $(\delta^{1,1}, \delta^{1,2})$ such that~\eqref{eq:gamma1} holds 
if and only if there exists a pair $(\delta^{2,1}, \delta^{2,2})$ such
that~\eqref{eq:gamma2} holds. 
To see this, recall that $\Delta = CA$ and $\gamma = C\beta$, for $C$ as above, so $\beta\in\NN A$ if and only if $\gamma\in\NN\Delta$. 
Further, by symmetry, it is enough to show that~\eqref{eq:gamma1} is equivalent to $\gamma\in\NN\Delta$. 
Notice that each $\delta \in\NN \Delta$ fulfills that $\delta_1 + \delta_2 = kl$, for some $l\in \NN$. 
This implies that if $\delta$ and $\tilde \delta$ are such that
$\delta_1 = \tilde \delta_1$, then $\delta = \tilde \delta + h\left[\begin{smallmatrix}0\\k\end{smallmatrix}\right]$ for some $h\in \ZZ$. 
If $(\delta^{1,1}, \delta^{1,2})$ satisfies~\eqref{eq:gamma1}, 
then $\delta_1^{1,1} = \gamma_1$ and $\delta_2^{1,1} \leq \gamma_2$. 
Hence, $\gamma = \delta_1 + h\left[\begin{smallmatrix}0\\k\end{smallmatrix}\right]$ for some $h \geq 0$. Since $\left[\begin{smallmatrix}0\\k\end{smallmatrix}\right] \in \Delta$,
$\gamma \in \NN\Delta$.
Conversely, if $\gamma \in \NN \Delta$, then because $\left[\begin{smallmatrix}k\\0\end{smallmatrix}\right] \in \Delta$, $\delta^{1,1} = \gamma$ and $\delta^{1,2} = \left[\begin{smallmatrix}k\\0\end{smallmatrix}\right]$ satisfies~\eqref{eq:gamma1}.
This establishes the proof of the claim. 

From the above calculation, we see that the failure of~\eqref{eq:gamma1} is equivalent to 
$\beta \in \left(\ZZ \fface+\NN A\right)\minus\NN A$.
Similarly, the failure of~\eqref{eq:gamma2} is equivalent to 
$\beta \in \left(\ZZ \lface+\NN A\right)\minus \NN A$.
Thus the desired conclusion for $\beta\in\ZZ^2\minus\cE_A$ follows from~\eqref{eqn:EAforCurves}.

Finally, if $\beta\in \cE_A$, we must show that the two finitely supported solutions recovered from 
\eqref{eqn:SolutionAlongPolar} and \eqref{eqn:SolutionAlongPolar2} are linearly independent. Expressions for the two Laurent polynomial solutions at rank-jumping parameters are known from \cite{CDD}, and they are identified by the negative supports of their monomials. One solution, which by necessity is obtained from \eqref{eqn:SolutionAlongPolar},  contains negative powers only of $x_1$. The other, by necessity is obtained from \eqref{eqn:SolutionAlongPolar2}, contains negative powers only of $x_n$, establishing the result.
\end{proof}

\section{Series solutions of hypergeometric systems on projective toric curves}
\label{sec:SeriesSolutions}

The results of Sections~\ref{sec:EM:genericResonant} and~\ref{sec:EM:nongenericResonant} show that, when beginning with analytic continuations of extended Euler--Mellin integrals, it is necessary to take parametric derivatives as in Section~\ref{sec:PD+sheaf} to complete a basis of the solution space of $H_A(\beta)$ at resonant parameters $\beta$. 
However, this technique is not sufficient to provide the statements of
Theorems~\ref{thm:main} and~\ref{thm:bigSheaf}, since it could be the case that along a resonant line, some solution in a basis obtained via parametric derivatives vanished at a point resonant with respect to both facets of $A$. 
To circumvent this issue, in this section we make use of series solutions of $H_A(\beta)$, with the goal of proving the following two results. 

\begin{theorem}
\label{thm:nonResonantLinesViaSeries}
If $L$ is a line contained in $\ResArr(A)$, 
then there exist
$\vol(A) = k$ functions 
that are linearly independent solutions of $H_A(\beta)$ for all $\beta \in L$ 
and are locally analytic on $(\CC^n \minus \Sigma_A) \times L$.
\end{theorem}

\pagebreak
\begin{theorem}
\label{thm:seriesLinIndep}
If $\bar{\beta}\in\ZZ^2$, 
there exists a neighborhood $\cW$ of
$\bar{\beta}$ in $\CC^2$ and $\vol(A)-1$ functions satisfying the following:
\begin{enumerate}
\vspace{-2.5mm}
\item they are solutions of $H_A(\beta)$ for all $\beta \in \cW$,
\item they are locally analytic for $(x,\beta)\in (\CC^n \minus \Sigma_A)\times\cW$,
\item they are linearly independent, and
\item for any fixed $\beta \in \cW$, the span of the specializations
  at $\beta$ of these functions does not contain any solution of $H_A(\beta)$ for which
  there exists an expression as a finitely supported series.
\end{enumerate}
\end{theorem}

To obtain these results, 
we draw heavily on techniques developed in the influential book~\cite{SST} by Saito, Sturmfels, and
Takayama, especially Sections~2.5,~3.2,~3.4, and~4.2 of
that text. 

A \emph{weight vector} $w \in \RR_{\geq 0}^n$ induces a partial
ordering $\preceq_w$ of the monomials in 
$\CC[\del]\defeq \CC[\del_1,\dots,\del_n]$ via 
$\del^\nu\preceq_w \del^{\nu'}$ if $w \cdot \nu \leq w \cdot \nu'$. 
If $w$ is generic, then the ideal $\ini_w(I_A)$ generated by 
the leading terms of the elements of $I_A$ with respect to 
$w$ is a monomial ideal. Moreover, given any monomial order 
$\prec$ in $\CC[\del]$, then there exists 
$w \in\RR_{\geq 0}^n$ such that 
$\ini_\prec(I_A) =  \ini_w(I_A)$.

Let $w \in \RR_{\geq 0}^n$ be a generic weight vector. The ideal 
\[
\fin_w(H_A(\beta)) \defeq \ini_w(I_A) + \<E-\beta\> \subseteq D
\]
is called the \emph{fake initial ideal of $H_A(\beta)$ with respect to
$w$}. If a monomial function $x^v$, with $v\in\CC^n$ is a solution of $\fin_w(H_A(\beta))$,
then $v$ is a \emph{fake exponent of $H_A(\beta)$ with respect to $w$}. 

If $w$ is a generic weight vector, $\preceq_w$ can be extended 
to $D$ via $x^u \del^\nu \preceq_w x^{u'}\del^{\nu'}$ whenever 
$-w\cdot u + w \cdot \nu \leq -w \cdot u'+w \cdot \nu'$. The \emph{initial ideal of $H_A(\beta)$ with respect to $w$}, denoted $\ini_{(-w,w)}H_A(\beta)$, is 
defined using this ordering, and it clearly contains 
$\fin_w(H_A(\beta))$.
The \emph{exponents} of $H_A(\beta)$ arise from the monomial solutions of $\ini_{(-w,w)}H_A(\beta)$, and therefore all exponents of $H_A(\beta)$ are fake exponents.

Let $I\subseteq \CC[\del]$ be a monomial ideal. Then $\del^r$ is a
\emph{standard monomial} of $I$ if $\del^r\notin I$.  
A pair $(r,\sigma)$, where $r\in\NN^n$ and $\sigma\subseteq
\{1,\dots,n\}$, is a \emph{standard pair} of $I$ if  
\begin{enumerate}
\vspace{-2mm}
\item $r_j = 0$ for $j\in\sigma$,
\item for each $u\in\NN^n$ with $u_j = 0$ for $j\notin\sigma$, the monomial $\del^r\del^u$ is a standard monomial of $I$,
\item for each $\ell\notin\sigma$, there exists $u\in\NN^n$ with $u_j = 0$ for $j\notin\sigma\cup\{\ell\}$ such that $\del^r\del^u\in I$.
\end{enumerate}
\vspace{-2mm}

By~\cite[Lemma~4.1.3]{SST}, fake exponents can be easily computed from
standard pairs. More explicitly, for each standard pair $(r,\sigma)$
of $\ini_w(I_A)$, if there exists $v \in \CC^n$ 
such that $Av =\beta$ and $v_i = r_i$ for $i \notin \sigma$ (such a
$v$ is unique if it exists), then $v$
is a fake exponent of $H_A(\beta)$ with respect to $w$, and all such
fake exponents arise this way.

Now fix $w$ representing the reverse lexicographic ordering 
$\del_1\prec\del_n\prec\del_{n-1}\prec\cdots\prec\del_3\prec\del_2$
for $I_A$. 
Then the standard pairs of $\ini_w(I_A)$ in this case 
are of the form $(r,\{1,n\})$ (there are $\vol(A) = k$ of these
\emph{top-dimensional} standard pairs)  
and $(r,\{1\})$. See~\cite[Section~4.2]{SST} for more details.

\begin{lemma}
\label{lemma:noIntegerDifferences}
Let $\beta \in \CC^2$ and let $r, r'\in\NN^n$ be such that $(r,\{1,n\})$ and $(r',\{1,n\})$ are top-dimensional standard pairs of $\ini_w(I_A)$. 
Denote by $v(\beta)$ and $v'(\beta)$
the fake exponents of $H_A(\beta)$ associated to $(r,\{1,n\})$ and $(r',\{1,n\})$, respectively. Then $v'(\beta) - v(\beta) \notin \ZZ^n$.
\end{lemma}
\begin{proof}
By contradiction, assume that  $v'(\beta)-v(\beta) = u \in \ZZ^n$. In particular,
$Au=0$, and therefore $\del^{u_+}-\del^{u_-} \in I_A$. The initial form
$\ini_w(\del^{u_+}-\del^{u_-} )$ belongs to the monomial ideal
$\ini_w(I_A)$, so we may assume that $\del^{u_+} \in
\ini_w(I_A)$. (If $\del^{u_-}$ is the lead term, then interchange the
  roles of $v'(\beta)$ and $v(\beta)$; if $w\cdot u = 0$, then both
  monomials must belong to the initial ideal.)

Note that $v'(\beta) + u_+ = v(\beta) + u_-$. Since the supports of $u_+$ and
$u_-$ are disjoint, $\del^{u_+} x^{v(\beta) + u_-} = x^{u_-}
\del^{u_+} x^{v(\beta)} = 0$, where the last equality follows from the fact that $x^{v(\beta)}$
is annihilated by all elements of $\fin_w(H_A(\beta))$. 
But then $\del^{u_+} x^{v'(\beta) + u_+} = 0$, 
so $v'(\beta)$ must have a strictly negative integer coordinate. 
By construction, all coordinates of $v'(\beta)$ indexed by $\{2,\dots,n-1\}$ are nonnegative integers. This implies that the only possible negative integer coordinates for $v'(\beta)$ are those indexed by $\{1,n\}$. 

Choose $\nu \in \RR^n$ with $\nu_i = 0$ for $i\neq 1,n$ such that the coordinates of $v'(\beta)+\nu$ indexed by $\{1,n\}$ are not integers. 
Then $v(\beta)+\nu \eqdef v(\beta+A\nu)$ and $v'(\beta)+\nu \eqdef
v'(\beta+A\nu)$ are the fake exponents of $H_A(\beta + A\nu)$ corresponding
to $(r,\{1,n\})$ and $(r',\{1,n\})$ respectively. Moreover $v'(\beta+A\nu) -
v(\beta+A\nu) = u$, so we may apply the previous argument to 
$v'(\beta+A\nu)$ and $v(\beta+A\nu)$ to conclude that $v'(\beta+A\nu)$
has a negative integer coordinate, a contradiction.
\end{proof}

The following auxiliary result is used to study the coefficients of a
hypergeometric series.
We use the Pochhammer notation $(\xi)_\ell \defeq \xi(\xi+1)\cdots (\xi+\ell-1)$ for
the $\ell$th ascending factorial of $\xi \in \CC$, and include a well-known result which will be useful to us.

\begin{lemma}
\label{lemma:coeffPochhammer}
The function $(\xi)_\ell$ for $\xi\in\CC$ is a polynomial in $\xi$, 
all of whose coefficients are nonnegative integers less than or 
equal to $\ell!$.
\qedhere
\end{lemma}

Let $(r,\{1,n\})$ be a top-dimensional standard pair of $\ini_w(I_A)$,
where $w$ is a generic weight vector representing the reverse
lexicographic term ordering for $I_A$ as before. Then
$v=(\xi,r_2,\dots,r_{n-1},\zeta)$ is
the fake exponent of $H_A(Av)$ corresponding to $(r,\{1,n\})$. 

Assume that $\xi,\zeta \notin \ZZ$. Then we can write the logarithm-free hypergeometric series 
\begin{equation}
\label{eqn:logFreeSeries}
\varphi_v(x,\xi,\zeta) \ \defeq \ 
\sum \frac{\prod_{u_i<0} \prod_{j=1}^{|u_i|} (v_i-j+1)}{\prod_{u_i>0}
  \prod_{j=1}^{u_i} (v_i+j)} x^{v+u},
\end{equation}
where the sum is over $u \in \ker_\ZZ(A)$ such that $u_i+r_i \geq 0$
for $i=2,\dots,n-1$. For each fixed $\xi,\zeta \notin \ZZ$, the
series~\eqref{eqn:logFreeSeries} is a solution of $H_A(Av)$
by~\cite[Theorem~3.4.14]{SST}. 

By Lemma~\ref{lemma:noIntegerDifferences}, series arising
from different top-dimensional standard pairs of $\ini_w(I_A)$ have disjoint supports and are therefore linearly independent.

\begin{theorem}
\label{thm:analyticSeries}
Assume that $A$ is of the form~\eqref{eqn:A}.  
Let $w\in\RR^n$ represent the reverse lexicographic ordering for $I_A$ as above. 
Fix a top-dimensional standard pair $(r,\{1,n\})$ of 
$\ini_w(I_A)$, and let $\scrU$ be a bounded open
subset of $(\CC\minus \ZZ)^2$. 
Then there exists an open set 
$W \subseteq \CC^n$ such that, for each 
$x = (x_1,\dots,x_n) \in W$, the series $\varphi_v(x,\xi,\zeta)$ of~\eqref{eqn:logFreeSeries} is an analytic function on $\scrU$.
\end{theorem}
\begin{proof}
In order to work with $\varphi_v(x,\xi,\zeta)$, we must
understand its summation range, which is related to $\ker_\ZZ(A)$. 
For each $2 \leq i \leq n-1$, there exist relatively prime positive integers 
$b_{1,i},b_{i,i},b_{n,i}$ such that 
$b_{1,i} + b_{n,i} = b_{i,i}$ and
$k_i b_{i,i} = k b_{n,i}$. Equivalently, 
$\del_i^{b_{i,i}} - \del_1^{b_{1,i}}\del_n^{b_{n,i}} 
\in I_A$. Form an $n\times(n-2)$ matrix $B$, 
with column indices $2,\dots,n-1$ for notational
convenience, whose $i$th column $B_i$ is the vector 
with first coordinate $-b_{1,i}$, $i$th coordinate equal to $b_{i,i}$, $n$th coordinate
equal to $-b_{n,i}$, and all other coordinates equal to zero. 

The lattice $\ZZ$-spanned by the columns of $B$ is contained in $\ker_\ZZ(A)$, but this containment may be strict; either way, $\ker_\ZZ(A)$ is a full rank sublattice of
$\bigoplus_{i=2}^{n-1} (\frac{1}{b_{i,i}})\ZZ \cdot B_i$.
Thus, the index set of $\varphi_v(x,\xi,\zeta)$ is in
bijection with
\begin{equation}
\label{eqn:summationIndex}
\left\{ (\ell_2,\dots,\ell_{n-1}) \in
\bigoplus_{i=2}^{n-1} \left(\frac{1}{b_{i,i}}\right)\ZZ \cdot B_i \ \Bigg|\ B 
\left[\begin{smallmatrix} \ell_2 \\ \vdots \\ \ell_{n-1} \end{smallmatrix}\right]
\in \ZZ^n \text{ and } \ell_i \geq -
\frac{r_i}{b_{i,i}} \text{ for } i \in\{2,\dots,n-1\} 
\right\}. 
\end{equation}

Note that the initial term of the binomial $\del_i^{b_{i,i}} -
\del_1^{b_{1,i}}\del_n^{b_{n,i}}$ is $\del_i^{b_{i,i}}$. Therefore
$\del_i^{b_{i,i}} \prod_{j=2}^{n-1} x_j^{r_j} =0$, which means that $r_i < b_{i,i}$. 
Thus $r_i/b_{i,i} < 1$. Hence if the columns of $B$ span $\ker_\ZZ(A)$ over $\ZZ$, then the
set~\eqref{eqn:summationIndex} equals $\NN^{n-2}$.

If the columns of $B$ do not $\ZZ$-span $\ker_\ZZ(A)$, then some of the indices $\ell_i$ in~\eqref{eqn:summationIndex} might be negative. 
However, since $b_{1,2},\dots,b_{1,n-1} > 0$, there are
only finitely many tuples $(\ell_2,\dots,\ell_{n-1}) \in
\bigoplus_{i=2}^{n-1} (1/b_{i,i})\ZZ \cdot B_i$ that are coordinatewise 
greater than or equal to 
$(-r_2/b_{2,2},\dots,-r_{n-1}/b_{n-1,n-1})$ 
and such that 
$b_{1,2}\ell_2+\cdots+b_{1,n-1} \ell_{n-1} < 0$. 
Similarly, there are only finitely many tuples 
$(\ell_2,\dots,\ell_{n-1})$ as above with 
$b_{n,2}\ell_2+\cdots+b_{n,n-1}\ell_{n-1} < 0$. 
These correspond to finitely many summands of $\varphi_v(x,\xi,\zeta)$, which we may separate out. If such summands exist, they lead to simple poles of $\varphi_v(x,\xi,\zeta)$ along some integer values of
$\xi$ or $\zeta$, which are outside the domain $\scrU$. 

Now consider the subseries of $\varphi_v(x,\xi,\zeta)$ indexed by the tuples $(\bar{\ell}_i)_{i \in \sigma}$ such that 
if $\sigma \subseteq\{2,\dots,n-1\}$ and 
$(\bar{\ell}_i)_{i \in \sigma}$ are such that
every tuple $(\ell_2,\dots,\ell_{n-1})$ in~\eqref{eqn:summationIndex}
satisfying $\ell_i = \bar{\ell}_i$ for $i \in \sigma$ also
satisfies $\ell_j \geq 0$ for all $j \notin \sigma$. 
 
If $\sigma = \varnothing$, the subseries
of $\varphi_v(x,\xi,\zeta)$ arising in this way has terms indexed by
the vectors corresponding to nonnegative tuples $\ell =
(\ell_2,\dots,\ell_{n-1})$. In any subseries corresponding to a
nonempty subset of $\{2,\dots,n-1\}$, the indices of summation are also 
nonnegative, but they live in a lower-dimensional cone. 
These subseries are easier to treat than the case $\sigma=\varnothing$ because there are more factorials in their denominators than in their numerators. 
Therefore it is enough to show that the subseries of $\varphi_v(x,\xi,\zeta)$ corresponding to
$\sigma=\varnothing$ is an analytic function of $\xi$ and $\zeta$. 
For notational convenience, and now without loss of generality, we assume
that the columns of $B$ span $\ker_\ZZ(A)$ over $\ZZ$, in which case
the subseries of $\varphi_v(x,\xi,\zeta)$ corresponding to $\sigma=\varnothing$ is in fact all
of $\varphi_v(x,\xi,\zeta)$, and it can be rewritten as 
\begin{align*}
\sum_{\ell_2,\dots,\ell_{n-1} = 0}^\infty
\frac{(-1)^{B^{(1)}\cdot\ell}
  (-\xi)_{B^{(1)}\cdot\ell}
  (-1)^{B^{(n)}\cdot\ell} (-\zeta)_{B^{(n)}\cdot\ell} }{
(r_2+1)_{b_{2,2}\ell_2} \cdots (r_{n-1}+1)_{b_{n-1,n-1}\ell_{n-1} }} 
\cdot \frac{x_2^{b_{2,2}\ell_2}\cdots x_{n-1}^{b_{n-1,n-1}\ell_{n-1}}}{x_1^{B^{(1)}\cdot\ell} x_n^{B^{(n)}\cdot\ell}}\cdot x^v
\end{align*}
where, by abuse of notation, $B^{(i)}\cdot\ell \defeq b_{i,2}\ell_2+\cdots+b_{i,n-1}\ell_{n-1}$.

Now take term-by-term absolute values with the assumption that $|x_1|,|x_n|>1$; also, since
dividing by a constant does not affect convergence, divide by $r_2!\cdots r_{n-1}!$. 
Moreover, as $x^v$ is clearly analytic in $\xi$ and $\zeta$, 
it is now enough to show that
\begin{align}
\label{eqn:ets}
\sum_{\ell_2,\dots,\ell_{n-1} = 0}^\infty 
\frac{(|\xi|)_{B^{(1)}\cdot\ell} (|\zeta|)_{B^{(n)}\cdot\ell}}{(b_{2,2}\ell_2)!\cdots (b_{n-1,n-1} \ell_{n-1})!} 
|x_2|^{b_{2,2}\ell_2} \cdots |x_{n-1}|^{b_{n-1,n-1}\ell_{n-1}} 
\end{align}
is analytic in $|\xi|$ and $|\zeta|$.
To do this, rewrite~\eqref{eqn:ets} as 
\[
\sum_{i,j = 0}^{\infty} \sum_{ \substack{
\ell_2,\dots,\ell_{n-1} \geq \ 0 \text{ such that }\\
b_{1,2}\ell_2+ \cdots+ b_{1,n-1} \ell_{n-1} = \ i \text{ and }\\
b_{n,2}\ell_2 +\cdots + b_{n,n-1}\ell_{n-1} = \ j
}} 
\frac{(|\xi|)_{i} (|\zeta|)_{j} }{(b_{2,2}\ell_2)!
    \cdots (b_{n-1,n-1} \ell_{n-1})!} |x_2|^{b_{2,2}\ell_2} \cdots
  |x_{n-1}|^{b_{n-1,n-1}\ell_{n-1}}.
\]
Since $b_{m,m} = b_{1,m}+b_{n,m}$ for $m\in\{2,\dots,n-1\}$, the previous
series is bounded above term by term by the series 
\begin{align*}
\sum_{i,j = 0}^{\infty} \sum_{ \substack{
\ell_2,\dots,\ell_n \geq \ 0 \text{ such that }\\
b_{1,2}\ell_2+ \cdots+ b_{1,n-1} \ell_{n-1} = \ i \text{ and }\\
b_{n,2}\ell_2 +\cdots + b_{n,n-1}\ell_{n-1} = \ j
}} {
\frac{(|\xi|)_i }{(b_{1,2}\ell_2)! \cdots (b_{1,n-1} \ell_{n-1})!} 
\frac{(|\zeta|)_j}{(b_{n,2}\ell_2)! \cdots (b_{n,n-1} \ell_{n-1})!} } \hspace{2.5cm} \\
\hspace{2.5cm}
\cdot
 |x_2|^{b_{1,2}\ell_2} \cdots  |x_{n-1}|^{b_{1,n-1}\ell_{n-1}} \cdot
|x_2|^{b_{n,2}\ell_2} \cdots  |x_{n-1}|^{b_{n,n-1}\ell_{n-1}},
\end{align*}
which is bounded above by the following series, obtained by adding more terms:
\begin{align*}
\sum_{i,j = 0}^{\infty} 
\sum_{ \substack{
\ell_2,\dots,\ell_{n-1} \geq \ 0\\
\ell_2+ \cdots+ \ell_{n-1} = \ i 
}} 
\sum_{ \substack{
\ell_2',\dots,\ell_{n-1}' \geq \ 0\\
\ell_2'+ \cdots+ \ell_{n-1}' = \ j 
}}
\frac{(|\xi|)_{i}}{\ell_2!\cdots \ell_{n-1}!} 
\frac{(|\zeta|)_{j} }{\ell_2'!\cdots \ell_{n-1}'!} 
|x_2|^{\ell_2}\cdots |x_{n-1}|^{\ell_{n-1}} 
\cdot 
|x_2|^{\ell_2'}\cdots |x_{n-1}|^{\ell_{n-1}'} .
\end{align*}
This is equal to the product of the two series 
\begin{align}
\label{eqn:seriesProduct1}
&\bigg[\sum_{i = 0}^{\infty} 
\sum_{ \substack{
\ell_2,\dots,\ell_{n-1} \geq \ 0\\
\ell_2+ \cdots+ \ell_{n-1} = \ i 
}}
\frac{(|\xi|)_{i}}{\ell_2!\cdots \ell_{n-1}!} 
|x_2|^{\ell_2}\cdots |x_{n-1}|^{\ell_{n-1}}
\bigg]
\\
\label{eqn:seriesProduct2}
\text{and}\quad 
&\bigg[
\sum_{j=0}^\infty
\sum_{ \substack{
\ell_2',\dots,\ell_{n-1}' \geq \ 0\\
\ell_2'+ \cdots+ \ell_{n-1}' = \ j 
}}
\frac{(|\zeta|)_{j} }{\ell_2'!\cdots \ell_{n-1}'!} 
|x_2|^{\ell_2'}\cdots |x_{n-1}|^{\ell_{n-1}'} 
\bigg].
\end{align}
Therefore, each factor of this product can be considered independently. Since they are
of the same form, we concentrate on the first one~\eqref{eqn:seriesProduct1}.
Apply Lemma~\ref{lemma:coeffPochhammer} to compare~\eqref{eqn:seriesProduct1} to
\begin{align*}
\sum_{N=0}^\infty \bigg[
\sum_{i\geq N} 
\sum_{ \substack{
\ell_2,\dots,\ell_{n-1} \geq \ 0\\
\ell_2+ \cdots+ \ell_{n-1} = \ i 
}} 
\frac{  i!}{\ell_2!\cdots \ell_{n-1}!} 
|x_2|^{\ell_2}&\cdots |x_{n-1}|^{\ell_{n-1}}  
\bigg]
|\xi|^N   \\
& = \sum_{N=0}^\infty 
\big[\sum_{i \geq N} (|x_2|+\cdots+|x_{n-1}|)^i \big]
|\xi|^N \\
& = 
\sum_{N=0}^\infty
\frac{(|x_2|+\cdots+|x_{n-1}|)^N }{1-(|x_2|+\cdots+|x_{n-1}|)} |\xi|^N,
\end{align*}
which converges for $\xi$ in a given bounded subset of $\CC \minus \ZZ$ if $|x_2|+\cdots+|x_{n-1}| \ll 1$. 
\end{proof}

\begin{remark}
\label{rmk:convergenceDomain-x}
In Theorem~\ref{thm:analyticSeries}, the domain of convergence, for
$(\xi,\zeta)$ in a bounded open subset of $\CC^2 \minus \ZZ^2$, of
the hypergeometric series under examination is of the form  
\[ 
W = \{x\in\CC^n\mid |x_1|,|x_n|>1,
|x_2|+\cdots+|x_{n-1}| \ll 1\}\subseteq\CC^n\minus\Sigma_A .
\] 
The containment above follows from the fact that $\Sigma_A$ is the
union of the $A$-discriminantal hypersurface and $\Var(x_1x_n)$. 
Since the Euler--Mellin integrals and their parametric
derivatives can be analytically continued to $\CC^n \minus \Sigma_A$,
these functions can be used 
to obtain analytic continuations of hypergeometric series, both in $x$
(in which case we extend to all of $\CC^n \minus \Sigma_A$) and
$(\xi,\zeta)$.
\endrk
\end{remark}

\begin{remark}
\label{rmk:extendingConvergenceDomain}
The proof of Theorem~\ref{thm:analyticSeries} used a reverse lexicographic term order with $\del_1$ 
as the lowest variable. As in the proof 
of~\cite[Theorem~4.2.4]{SST},~\cite[Theorem~2.5.14]{SST} can be used to rewrite the summation set of the series 
$\varphi_v(x,\xi,\zeta)$ from~\eqref{eqn:logFreeSeries} 
using the dual cone of the Gr\"obner cone associated to this term order. From this we conclude that any $
u\in\ker_{\ZZ}(A)$ that indexes a term of 
$\varphi_v(x,\xi,\zeta)$ must satisfy $u_1\leq0$. 
This implies that $\varphi_v(x,\xi,\zeta)$ is locally analytic 
for $(\xi,\zeta) \in \CC \times (\CC \minus \ZZ)$, 
since none of the summands involve denominators that depend on $\xi$.

Theorem~\ref{thm:analyticSeries} could also be proved using a reverse lexicographic term order with $\del_n$ now as the lowest variable (and $\del_1$ as the next lowest). 
The form of the convergence domain in $x$ clearly does not change, so
in this case, we obtain series that are locally
analytic for $(\xi,\zeta) \in (\CC \minus \ZZ) \times \CC$.
\endrk
\end{remark}

We are now prepared to prove Theorems~\ref{thm:nonResonantLinesViaSeries} and~\ref{thm:seriesLinIndep}, which was the goal of this
section.  

\begin{proof}[Proof of Theorem~\ref{thm:nonResonantLinesViaSeries}]
Use the same assumptions and notation as in the proof of
Theorem~\ref{thm:analyticSeries}, and let $\prec$ be the reverse
lexicographic term order satisfying $\del_1\prec \del_n \prec \del_2
\prec \cdots \prec \del_{n-1}$.
Suppose first that $L = \{ \mu + \lambda \fface \mid \lambda \in \CC
\}$, where $\mu \in \ZZ^2$ is fixed. 
If $(r,\{1,n\})$ is a top-dimensional standard pair of
$\ini_\prec(I_A)$, then the corresponding fake exponent of  $H_A(\mu +
\lambda \fface )$ is 
\[
(\lambda + (\mu-Ar)_1-(1/k)(\mu-Ar)_2, r,(1/k)(\mu-Ar)_2),
\] 
where we have abused notation by using $r$ as a
vector in $\NN^n$ whose first and last coordinates are zero and also
as a vector in $\NN^{n-2}$. 
Note that the last coordinate of this fake exponent does not
depend on $\lambda$. 

By the proof of~\cite[Theorem~4.2.4]{SST}, $H_A(\mu + \lambda \fface )$
has at most $\vol(A)+1$ fake exponents: the ones corresponding to the
top-dimensional standard pairs, along with possibly one other arising from
a standard pair of the form $(r',\{1\})$. Such a standard
pair gives rise to a fake exponent of $H_A(\mu + \lambda \fface )$ if
and only if $\mu_2 = (Ar')_2$, a condition which is independent of
$\lambda$. Thus, if there is a fake exponent associated to a non-top-dimensional standard pair for one parameter in $L$, then it is a fake exponent for every parameter in $L$. 
In this case, the last coordinate of this fake exponent is also independent of $\lambda$. 

If the line $L$ is such that all fake exponents of $H_A(\beta)$ correspond to
top-dimensional standard pairs, then they each give rise to a series solution of $H_A(\mu + \lambda \fface )$ for every $\lambda \in \CC$, and
all of these solutions are locally analytic functions of $\lambda$ by
Theorem~\ref{thm:analyticSeries} and
Remark~\ref{rmk:extendingConvergenceDomain}, since the last coordinate
of these fake exponents is independent of
$\lambda$. Remark~\ref{rmk:convergenceDomain-x} provides locally
analytic extensions to a domain in $x$ common to all $\lambda$, namely
$\CC^n\minus \Sigma_A$.
These functions
are linearly independent by~\cite[Lemma~2.5.6(2)]{SST}. 

If the line $L$ is such that there exists a fake exponent
corresponding to a standard pair $(r',\{1\})$, then by the proof
of~\cite[Theorem~4.2.4]{SST}, there exists a top-dimensional standard
pair $(r,\{1,n\})$ such that $ A(r-r')$ is an integer combination of
the first and last columns of $A$. That proof also asserts 
that the fake exponent associated to $(r,\{1,n\})$ does not 
give rise to a series solution of 
$H_A(\mu + \lambda \fface )$ unless 
$\mu + \lambda\fface \in \cE_A$. This means that, if 
$\mu + \lambda\fface
\notin \cE_A$, then the $\vol(A)-1$ top-dimensional standard
pairs different from $(r,\{1,n\})$ induce series solutions of 
$H_A(\mu + \lambda \fface )$.
Moreover, if  $\mu + \lambda\fface
\in \cE_A$, 
every fake exponent gives rise to a solution of the corresponding
hypergeometric system.
Now the same argument as above applied to the series arising from
$(r',\{1\})$ and the $\vol(A)-1$ top-dimensional standard pairs
different from $(r,\{1,n\})$
yields $\vol(A)$ locally analytic functions on $(\CC^n \minus \Sigma_A)
\times L$ that are linearly
independent solutions of $H_A(\beta)$ for $\beta \in L$.

For the case that $L =\{ \mu + \lambda \lface \mid \lambda
\in \CC\}$, where $\mu
\in \ZZ^2$ is fixed, the proof is essentially the same as the previous
one, except that the term order must be
changed. In this case, use $\prec'$, the reverse lexicographic term
order with $\del_n \prec' \del_1 \prec'\del_2 \prec' \cdots \prec'
\del_{n-1}$. This switches the roles of the first and last coordinates
in the proof of the previous case. In other words, the first 
coordinate of each fake exponent is now independent of $\lambda$, and the
condition for the existence of a fake exponent associated to a
standard pair $(r',\{n\})$ is $k(\mu-Ar')_1=(\mu-Ar')_2$, which is also 
independent of $\lambda$. Now Theorem~\ref{thm:analyticSeries}
and Remarks~\ref{rmk:extendingConvergenceDomain}
and~\ref{rmk:convergenceDomain-x} again imply the desired result. 
\end{proof}

\begin{proof}[Proof of Theorem~\ref{thm:seriesLinIndep}]
Let $w\in\RR^n$ represent the reverse lexicographic ordering for $I_A$ as above. 
Let $\bar{\beta} \in \ZZ^2$, and 
consider the fake exponents of $H_A(\bar{\beta})$ corresponding to
top-dimensional standard pairs of $\ini_w(I_A)$. 
By Lemma~\ref{lemma:noIntegerDifferences}, their 
pairwise differences are not integer vectors, which means that there is at most one such fake exponent with all integer coordinates. Thus there exist 
$\vol(A)-1$ fake exponents with at least one noninteger
coordinate. Those coordinates can only be the first and last, since the others are forced to be nonnegative integers by construction. 
However, since $\bar{\beta} \in \ZZ^2$, if either of the first or last coordinates of
a fake exponent of $H_A(\bar{\beta})$ is noninteger, then so is the
other. 

Applying Theorem~\ref{thm:analyticSeries} to each of
the associated $\vol(A)-1$ series implies that they are analytic
functions of $\beta$ in a neighborhood $\cW$ of $\bar{\beta}$, for $x$
in an open subset of $\CC^n \minus \Sigma_A$.
(Recall that $\zeta$ and $\xi$ in Theorem~\ref{thm:analyticSeries} are linear functions of $\beta_1$ and $\beta_2$.)
By Lemma~\ref{lemma:noIntegerDifferences}, series arising from different top-dimensional standard pairs of $H_A(\beta)$ have disjoint supports and are therefore linearly independent.
Moreover, the fact that the supports of these series are disjoint implies that no linear combination of them can have any cancelled terms. In particular, such functions cannot contain a finitely supported series in their span.

Now Remark~\ref{rmk:convergenceDomain-x} provides analytic
continuations of these series to $x$ in $\CC^n \minus \Sigma_A$.
\end{proof}

\begin{remark}
\label{rmk:seriesObstacles}
The proofs in this section rely very strongly on the fact that $A$ is
of the form~\eqref{eqn:A}, especially on the fine control over the standard pairs of
reverse lexicographic initial ideals of $I_A$ that is available in
this case. In general, the potential existence of logarithmic series
solutions of $H_A(\beta)$ makes it much more difficult to decide whether a
fake exponent of $H_A(\beta)$ is a true exponent.
\endrk
\end{remark}

\section{Parametric behavior of \texorpdfstring{$A$}{A}-hypergeometric functions}
\label{sec:sheafProof}

In this section, we combine the results in the article to prove
Theorems~\ref{thm:main} and~\ref{thm:bigSheaf} and then discuss how hypergeometric solutions at
rank jumps are connected to local cohomology. We also include an example to illustrate these ideas.

\begin{proof}[Proof of Theorem~\ref{thm:main}]
The existence of $\vol(A)$ linearly independent functions that are
locally analytic for
$(x,\beta) \in (\CC^n \minus \Sigma_A) \ \times (\CC^2 \minus \ResArr(A))$ and span the
solution space of $H_A(\beta)$ for any $\beta \in \CC^2 \minus
\ResArr(A)$ can be proved using residue integrals as in~\cite{GKZ90}
or analytic continuations of extended Euler--Mellin integrals as in
Theorem~\ref{thm:EMindep}.

For each line $L$ in $\ResArr(A)$, Theorem~\ref{thm:nonResonantLinesViaSeries}
provides $\vol(A)$ linearly independent solutions of $H_A(\beta)$ that
vary locally analytically
for $\beta \in L$ and span the solution space of $H_A(\beta)$ if
$\beta \notin \cE_A$.
Note that
if $L_1$ and $L_2$ are two resonant lines that meet outside
$\cE_A$, then the bases of solutions provided by 
Theorem~\ref{thm:nonResonantLinesViaSeries} along $L_1$ and $L_2$ are
not necessarily the same when specialized at the intersection point
$L_1 \cap L_2$, but they do span the same solution space. Thus, at
such an intersection, every solution of the corresponding
hypergeometric system can be analytically deformed to a hypergeometric
function both along $L_1$ and along $L_2$.

Now let $\bar{\beta} \in \cE_A$, and let
$L_1$ and $L_2$ be the two (polar) resonant lines whose intersection
is $\bar{\beta}$. By Theorem~\ref{thm:genericResPolar}, for
$i=1,2$, there is a finitely supported solution of $H_A(\beta)$ that
varies locally analytically for $\beta \in L_i$. Further, by
Theorem~\ref{thm:nongenericResPolarEM}, these two finitely supported
series specialize to linearly independent functions at $\bar{\beta}$.
Moreover, Theorem~\ref{thm:seriesLinIndep} provides a neighborhood
$\mathcal{W}$ of $\bar{\beta}$ and $\vol(A)-1$ linearly independent
solutions of $H_A(\beta)$ that are locally analytic for $(x,\beta) \in
(\CC^n \minus \Sigma_A) \times
\mathcal{W}$, whose span never contains those finitely supported series.

Thus we have a set of $\vol(A)$ linearly independent solutions of
$H_A(\beta)$ 
that are locally analytic for $\beta \in L_1 \cap \mathcal{W}$, and another
set of  $\vol(A)$ linearly independent solutions of $H_A(\beta)$
that are locally analytic for $\beta \in L_2 \cap \mathcal{W}$. When we
specialize to $\bar{\beta}$, the union of these two sets spans a 
space of dimension $\vol(A)+1$, which is thus the whole solution space
of $H_A(\bar{\beta})$.

Note that each of the finitely supported series mentioned above can
only be analytically deformed to a solution of $H_A(\beta)$ along one of the lines $L_1$ or $L_2$, since otherwise, 
taking a limit $\beta \to \bar{\beta}$, we would
obtain $\rank(H_A(\bar{\beta})) = \vol(A)$. 
\end{proof}

\begin{proof}[Proof of Theorem~\ref{thm:bigSheaf}]
Consider a reverse lexicographic term order $\prec$ 
with $\del_1$ lowest, followed by $\del_n$. There are finitely many resonant lines
corresponding to $\fface$ such that a non-top-dimensional standard pair of
$\ini_\prec(I_A)$ provides a fake exponent 
of $H_A(\beta)$ with respect to $\prec$ along that line.

Let $L = \{ \mu + \lambda \fface \mid \lambda \in \CC\}$ be such that all
fake exponents of $H_A(\beta)$ with respect to $\prec$ along $L$ arise
from top-dimensional standard pairs, and fix $\bar{\beta} \in L$. Then there exists a small open neighborhood $\mathcal{V} \subseteq \CC^2$ of
$\bar{\beta}$ such that, for all $\beta \in \mathcal{V}$, the
top-dimensional fake exponents of $H_A(\beta)$ with respect to $\prec$ give rise to series
that span the solution space of $H_A(\beta)$ and are analytic on
$\mathcal{V}$ (and the domain of convergence in $x$ can then be extended
to $\CC^n \minus \Sigma_A$ by analytic continuation). Here we have
used Theorem~\ref{thm:analyticSeries} and Remark~\ref{rmk:convergenceDomain-x}.
This allows us to add to $\CC^2\minus \ResArr(A)$ all but
finitely many of the lines that are translates of the span of the face $\left\{\fface\right\}$.

The same argument, using a reverse lexicographic term ordering with
$\del_n$ lowest (followed by $\del_1$), applies to all but finitely
many lines that are resonant with respect to $\left\{\lface\right\}$. 

Let $\sL$ denote the union of the finitely many lines (some resonant
with respect to $\left\{\fface\right\}$, others resonant with respect
to $\left\{\lface\right\}$) not covered by the above arguments. Note
that all of those lines are polar, and that $\sL$ contains all (polar)
resonant lines that meet $\cE_A$. We have shown that the solutions of $H_A(\beta)$ 
vary locally analytically for $(x,\beta) \in (\CC^n \minus \Sigma_A)
\times (\CC^2 \minus \sL)$.

The behavior of the solutions of $H_A(\beta)$ along $(\CC^n \minus
\Sigma_A) \times \sL$ is already controlled by Theorem~\ref{thm:main},
since $\sL \subseteq \ResArr(A)$.
\end{proof}

In the proof of Theorem~\ref{thm:bigSheaf}, the polar
lines that meet at points in $\cE_A$ must lie in $\sL$, but there
might be lines in $\sL$ that do not 
contain any rank-jumping parameters. The most elegant possible statement
regarding variation of solutions with respect to the parameters would
involve also removing these lines from $\sL$, but we 
do not currently have an argument that makes this possible.  

In the proof of Theorem~\ref{thm:main}, we saw how the finitely supported series solutions along resonant lines of an $A$-hypergeometric system interact to form rank
jumps. The connection between such solutions and $\cE_A$ in
the case of curves is not new, as it was shown in~\cite{CDD} that
$\rank(H_A(\beta)) > \vol(A)$ if and only if $H_A(\beta)$ has two
linearly independent Laurent polynomial solutions. 
On the other hand, there is also a strong connection between rank jumps and the graded structure of the local cohomology of $\CC[\del]/I_A$ at the maximal ideal (see~\cite{MMW} for the most general version).

We now make explicit the connection between solutions of $H_A(\beta)$ at rank-jumping parameters $\beta$ and the first local cohomology module of $\CC[\del]/I_A$. The ring $\CC[\del]/I_A$ (and each of its local cohomology modules with respect to the ideal $\<\del\>$) carries a natural grading by $\ZZ^2$, where $\deg(\del_i)$ is given by the $i$th column of $A$. 
For $A$ as in~\eqref{eqn:A}, the local cohomology modules $H_{\<\del\>}^\bullet(\CC[\del]/I_A)$ can be computed from the Ishida complex: 
\begin{equation}
\label{eqn:ishida}
H_{\<\del\>}^\bullet\left(\frac{\CC[\del]}{I_A}\right) = 
H^\bullet\left[\,
0\to\frac{\CC[\del]}{I_A}\stackrel{\ \delta^0}{\longrightarrow}
\left(\frac{\CC[\del]}{I_A}\right)_{\del_1}\oplus\left(\frac{\CC[\del]}{I_A}\right)_{\del_n}
\stackrel{\ \delta^1}{\longrightarrow}
\left(\frac{\CC[\del]}{I_A}\right)_{\del_1\del_n}
\to 0\,
\right], 
\end{equation}
where ${\CC[\del]}/{I_A}$ lies in cohomological degree $0$ (see, for
instance,~\cite[Section~13.3.1]{miller-sturmfels}). 
It thus follows that 
\[
\left\{\alpha\in\ZZ^2\;\bigg\vert\; \left[H_{\<\del\>}^1\left(\frac{\CC[\del]}{I_A}\right)\right]_\alpha\neq 0\,\right\} 
= \cE_A = \left[\left(\ZZ \fface +\NN A\right) \cap
\left(\ZZ \lface +\NN A\right)\right] 
\minus \NN A
\subseteq \RR_{\geq 0}A.
\]

Since $\dim_\CC \left({\CC[\del]}/{I_A}\right)_{\del_i} = 1$ for any $i$, note also that $\dim_\CC [H_{\<\del\>}^1(\CC[\del]/I_A)]_\alpha \leq 1$ for all $\alpha\in\ZZ^2$. 
We show now use Theorem~\ref{thm:nongenericResPolarEM} to recover cocycle generators for the nonzero graded components of $H_{\<\del\>}^1(\CC[\del]/I_A)$. 

\begin{theorem}
\label{thm:localcohom}
Let $\beta\in\cE_A$, and let $\varphi(x)$ and $\varphi'(x)$ respectively denote the finitely supported solutions of $H_A(\beta)$ obtained from \eqref{eqn:SolutionAlongPolar} and \eqref{eqn:SolutionAlongPolar2} in Theorem~\ref{thm:nongenericResPolarEM}. 
Then if $x^v$ and $x^{v'}$ are monomials respectively appearing with nonzero coefficients in those solutions, then 
\[
\left(
\del^v,\del^{v'}
\right)\in \left(\frac{\CC[\del]}{I_A}\right)_{\del_1}\oplus\left(\frac{\CC[\del]}{I_A}\right)_{\del_n}
\]
forms a nonzero cocycle generator of $H_{\<\del\>}^1(\CC[\del]/I_A)$. 
\end{theorem} 
\begin{proof}
Since $\beta\in \cE_A = \left[\left(\ZZ \fface +\NN A\right) \cap
\left(\ZZ \lface +\NN A\right)\right] 
\minus \NN A$, 
then neither $v$ nor $v'$ belongs to $\NN^n$. In particular, $(\del^{v},\del^{v'})$ is not in the image of $\delta^0$ in~\eqref{eqn:ishida} because if it were, then $v=v'\in\NN^n$. 

To complete the proof, it is enough to show that 
$(\del^{v},\del^{v'})\in\ker \delta^1$ from~\eqref{eqn:ishida}. 
Note that $Av=Av'=\beta$ and $v,v' \in \ZZ^n$. 
Further, all coordinates of
$v$ are nonnegative except for the first one, and all coordinates of
$v'$ are nonnegative except for the last one. Thus
$A(v-v_1e_1-v'_ne_n) = A(v'-v_1e_1-v'_ne_n)$, where
$e_1$ and $e_n$ are standard basis vectors of $\ZZ^n$, which implies
that $\del_n^{v'_n}\prod_{i=2}^n \del_i^{v_i} -
\del_1^{v_1}\prod_{i=1}^{n-1}\del_i^{v'_i} \in I_A$. Equivalently, $\del^{v}-\del^{v'} \in (I_A)_{\del_1\del_n}$, which yields the desired containment. 
\end{proof}
 
\begin{example}
\label{ex:0134sheaf}
Let us demonstrate the behavior of solutions of $H_A(\beta)$, as described in Theorems~\ref{thm:main} and~\ref{thm:bigSheaf}, when 
$A = \left[\begin{smallmatrix}1&1&1&1\\0&1&3&4\end{smallmatrix}\right]$. 
For the stratum of nonresonant parameters, $\CC^2\minus\ResArr(A)$, 
Theorem~\ref{thm:EMindep} guarantees that there are four distinct
Euler--Mellin integrals that provide a basis of solutions, which vary
locally analytically in $\beta$. 

For a resonant line $L= \mu+\lambda\fface$ for $\mu\in\ZZ^2$, $\lambda\in\CC$, we use a revlex ordering $\prec$ where $\del_1$ is the least variable. In this case, the top-dimensional standard pairs of $\ini_\prec(I_A)$ are $(r,\{1,4\})$, where $r$ is one of 
\begin{align*}
(0,0,0,0), \quad 
(0,1,0,0), \quad 
(0,0,1,0), \quad
\text{or}\quad 
(0,0,2,0),
\end{align*}
and there is one lower dimensional standard pair as well, which is $\left( (0,2,0,0), \{1\} \right)$. 
As long as $\mu_2\neq 2$, then the four top-dimensional standard pairs
provide $\vol(A)$ series solutions that vary locally analytically along $L$ with respective exponents
\begin{align*}
v_0 &\defeq 
\left( \beta_1-\frac{\beta_2}{4},0,0,\frac{\beta_2}{4} \right), \\
v_1 &\defeq
 \left( \beta_1-\frac{\beta_2+3}{4},1,0,\frac{\beta_2-1}{4} \right), \\
v_2 &\defeq
\left( \beta_1-\frac{\beta_2+1}{4},0,1,\frac{\beta_2-3}{4} \right),\quad\text{and} \\
v_3 &\defeq
\left( \beta_1-\frac{\beta_2-2}{4},0,2,\frac{\beta_2-6}{4} \right).
\end{align*}
Along the line given by $\beta_2 = \mu_2 = 2$, then the exponent $v_3$ above is replaced by 
$\left( \beta_1-2,1,0,0 \right)$. 

Similarly, for a resonant line $L= \mu+\lambda\4face$ for $\mu\in\ZZ^2$, $\lambda\in\CC$, we use a revlex ordering $\prec'$ where $\del_n$ is the least variable. In this case, the top-dimensional standard pairs of $\ini_{\prec'}(I_A)$ are $(r,\{1,4\})$, where $r$ is one of 
\begin{align*}
(0,0,0,0), \quad 
(0,1,0,0), \quad 
(0,2,0,0), \quad
\text{or}\quad 
(0,3,0,0),
\end{align*}
and there are three lower dimensional standard pairs as well, which are 
\begin{align*}
\left((0,0,1,0), \{4\} \right), \quad 
\left((0,0,2,0), \{4\} \right), \quad 
\text{and}\quad 
\left((1,0,1,0), \{4\} \right).
\end{align*} 
When $4\mu_1-\mu_2\notin\{1,2,5\}$, then the four top-dimensional
standard pairs provide $\vol(A)$ series solutions that vary locally analytically along $L$ with respective exponents
\begin{align*}
v_0 &\defeq 
\left( \beta_1-\frac{\beta_2}{4},0,0,\frac{\beta_2}{4} \right), \\
v_1 &\defeq
 \left( \beta_1-\frac{\beta_2+3}{4},1,0,\frac{\beta_2-1}{4} \right), \\
v'_2 &\defeq
\left( \beta_1-\frac{\beta_2+6}{4},2,0,\frac{\beta_2-2}{4} \right),\quad\text{and} \\
v'_3 &\defeq
\left( \beta_1-\frac{\beta_2-9}{4},3,0,\frac{\beta_2-3}{4} \right).
\end{align*}
Along the lines given by $4\beta_1-\beta_2 = 4\mu_1-\mu_2$ equal to $1,2,$ or $5$, the respective exponents 
$v'_3, v'_2$, and $v'_3$ above are replaced by  
\[
\left( 0,0,1,\beta_1-1 \right),\quad  
\left( 0,0,2,\beta_1-2 \right), \quad 
\text{and}\quad 
\left( 1,0,1,\beta_1-1 \right).
\] 

For any $\mu \in \ZZ^2$, if $L$ is a line resonant with respect to either
facet that contains $\mu$, the exponent corresponding to the finitely
supported solution along $L$ varies with the value of $\lambda$ modulo
$4$.  

The unique rank-jumping parameter $\beta =
\left[\begin{smallmatrix}1\\2\end{smallmatrix}\right]$ belongs to two
special resonant lines, in the sense that both involved discarding a
(top-dimensional) fake exponent to instead use a lower-dimensional standard pair. However, at this point, the discarded fake exponent from one line agrees with the exponent on the other line coming from the lower-dimensional standard pair. Thus, while the discarded fake exponents are not valid for the whole line, they are indeed valid at any point in $\cE_A$. 

For a resonant parameter $\beta\in L_1\cap L_2$ that belongs the
intersection of two resonant lines but is not rank-jumping, the story
is simpler, as was explained in the proof of Theorem~\ref{thm:main}.

In summary, these computations show that for this example, the $\sL$
in Theorem~\ref{thm:bigSheaf} is the union of four polar lines. This
set is depicted in Figure~\ref{fig:0134nontopLines}. 

\begin{figure}[t]
\includegraphics[width=60mm]{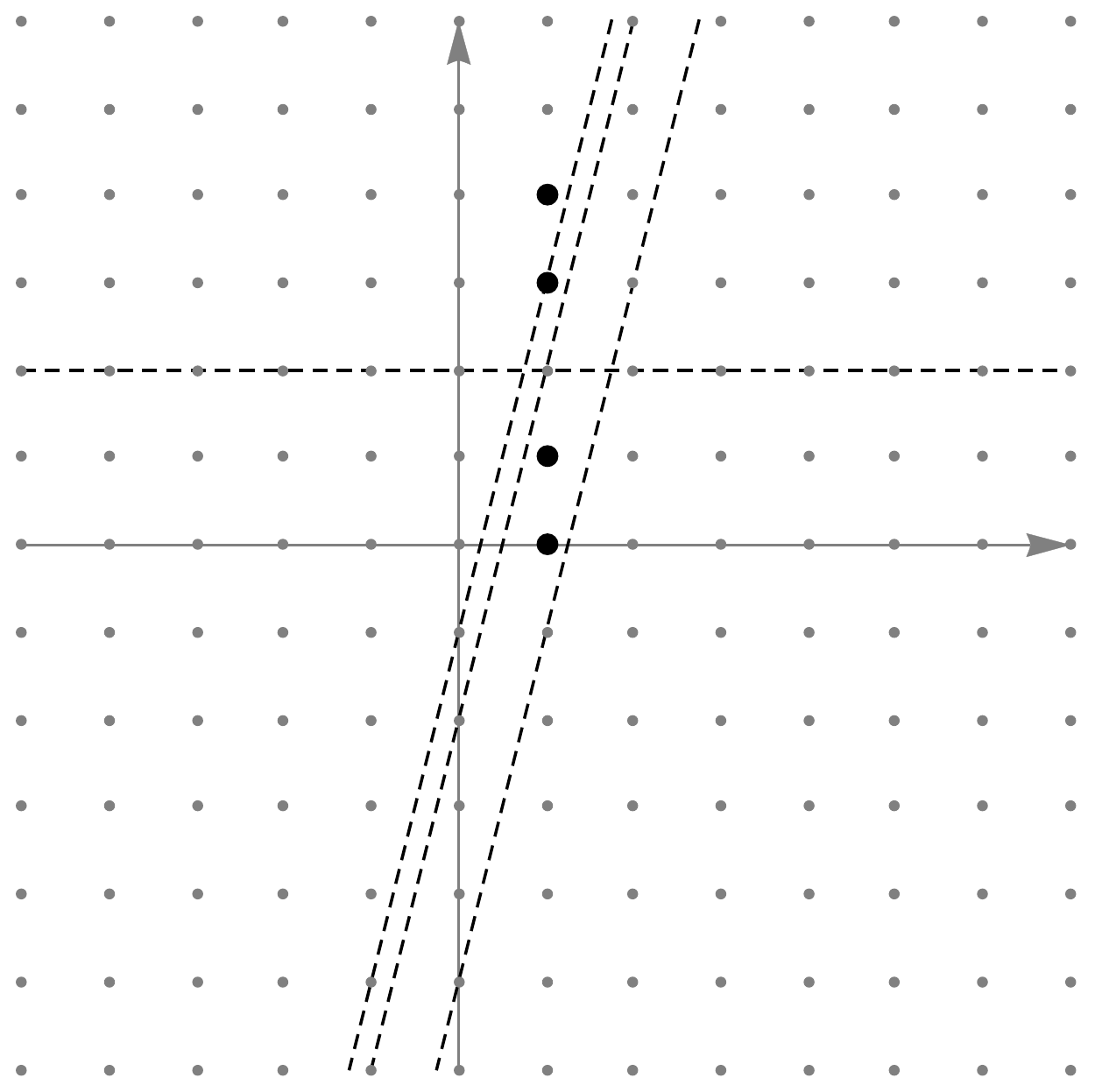}
\caption{The line arrangement $\sL$ given by the proof of Theorem~\ref{thm:bigSheaf} for $A =$\usebox{\smlMat}.
}
\label{fig:0134nontopLines}
\end{figure}

To improve the situation even more, there is another term order $\prec''$, with $\del_4$ lowest, which gives an initial ideal $\ini_{\prec''}(I_A)$ with top-dimensional standard pairs $(r,\{1,4\})$
where $r$ equals
\[
(0,0,0,0), \quad 
(0,1,0,0), \quad 
(0,2,0,0), \quad 
\text{or}\quad 
(0,0,1,0), 
\]
and one additional standard pair $\left( (0,0,2,0) ,\{4\}\right)$. Thus, when it comes to polar lines corresponding to $\left\{\lface\right\}$, it is only necessary to include the line $4\beta_1-\beta_2 = 2$. Using the methods in Theorem~\ref{thm:bigSheaf}, the line arrangement $\sL$ can thus be replaced by $\sL'$, which is the two lines $ \{ \beta \mid \beta_2=2\}$ and $\{ \beta \mid
4\beta_1-\beta_2 = 2 \}$ that meet at the rank-jumping parameter $\left[\begin{smallmatrix}1\\2\end{smallmatrix}\right]$. 
\endrk
\end{example}

\raggedbottom
\providecommand{\MR}{\relax\ifhmode\unskip\space\fi MR }
\providecommand{\MRhref}[2]{%
  \href{http://www.ams.org/mathscinet-getitem?mr=#1}{#2}
}
\providecommand{\href}[2]{#2}
\end{document}